\documentclass[11pt, a4paper, twoside]{amsart}
\usepackage{amssymb}
\usepackage{amsfonts, graphicx, color}
\usepackage{mathrsfs}

\theoremstyle{plain}
\newtheorem{theorem}{Theorem}

\newtheorem{corollary}{Corollary}

\newtheorem{lemma}{Lemma}

\numberwithin{equation}{section}
\begin{document}
\title[Bilinear Riesz means on the Heisenberg group]{Bilinear Riesz means on
the Heisenberg group}
\author{Heping Liu}
\address{School of Mathematical Sciences, Peking University, Beijing, 100871,
P. R. China, }
\email{hpliu@math.pku.edu.cn}
\author{Min Wang}
\address{School of Mathematical Sciences, Peking University, Beijing, 100871,
P. R. China, }
\email{wangmin09150102@163.com}
\date{December 10, 2017}

\begin{abstract}
In this article, we investigate the bilinear Riesz means $S^{\alpha }$
associated to the sublaplacian on the Heisenberg group. We prove that the
operator $S^{\alpha }$ is bounded from $L^{p_{1}}\times L^{p_{2}}$ into $
L^{p}$ for $1\leq p_{1}, p_{2}\leq \infty $ and $1/p=1/p_{1}+1/p_{2}$ when $
\alpha $ is large than a suitable smoothness index $\alpha (p_{1},p_{2})$.
There are some essential differences between the Euclidean space and the
Heisenberg group for studying the bilinear Riesz means problem. We make use
of some special techniques to obtain a lower index $\alpha (p_{1},p_{2})$.
\end{abstract}

\maketitle

\allowdisplaybreaks

\section{Introduction}

The bilinear Bochner-Riesz means problem originates from the study of the
summability of the product of two $n$-dimensional Fourier series. This leads
to the study of the $L^{p_{1}}\times L^{p_{2}}\rightarrow L^{p}$ boundedness
of the bilinear Bochner-Riesz multiplier
\begin{equation*}
B^{\alpha}(f,g)(x)=\int \int_{\left\vert \xi \right\vert ^{2}+\left\vert
\eta \right\vert ^{2}\leq 1} (1-\left\vert \xi \right\vert ^{2}-\left\vert
\eta \right\vert ^{2})^{\alpha }\widehat{f}(\xi ) \widehat{g}(\eta )e^{2\pi
ix\cdot (\xi +\eta )}d\xi d\eta .
\end{equation*}
Here $x\in \mathbb{R}^{n}$, $f$, $g$ are functions on $\mathbb{R}^{n}$ and $
\widehat{f}$, $\widehat{g}$ are their Fourier transforms. Bernicot et al.
\cite{Bern} gave a comprehensive study on the $L^{p_{1}}\times
L^{p_{2}}\rightarrow L^{p}$ boundedness of the operator $B^{\alpha}$.
Inspired by their work, we investigate the corresponding problem on the
Heisenberg group.

Strichartz \cite{Strich, Strich2} developed the harmonic analysis on the
Heisenberg group as the spectral theory of the sublaplacian. We can define
the bilinear Riesz means in terms of the spectral decomposition of the
sublaplacian. Let
\begin{equation*}
\mathcal{L}f=\int_{0}^{\infty }\lambda P_{\lambda }f\,d\mu (\lambda ).
\end{equation*}
be the spectral decomposition of the sublaplacian $\mathcal{L}$. The
bilinear Riesz means associated to the sublaplacian $\mathcal{L}$ is defined
by
\begin{equation*}
S^{\alpha }(f,g)=\int_{0}^{\infty }\int_{0}^{\infty }\left( 1- \lambda
_{1}-\lambda _{2}\right) _{+}^{\alpha }P_{\lambda _{1}}fP_{\lambda
_{2}}g\,d\mu (\lambda _{1})d\mu (\lambda _{2}).
\end{equation*}
As same as the Euclidean case, we hope to obtain a lower smoothness index $
\alpha (p_{1},p_{2})$ such that the operator $S^{\alpha }$ is bounded from $
L^{p_{1}}\times L^{p_{2}}$ into $L^{p}$ for $1\leq p_{1}, p_{2} \leq \infty $
and $1/p=1/p_{1}+1/p_{2}$ when $\alpha > \alpha (p_{1},p_{2})$.

There are some essential differences between the Euclidean space and the
Heisenberg group for studying the bilinear Riesz means problem. Firstly, the
kernel of the bilinear Bochner-Riesz operator $B^{\alpha}$ on $\mathbb{R}
^{n} $ coincides with the kernel of the Bochner-Riesz operator on $\mathbb{R}
^{2n} $. The pointwise estimate of this kernel is well known and gives a
basic result: $B^{\alpha}$ is bounded from $L^{p_{1}}\times L^{p_{2}}$ into $
L^{p}$ when $\alpha > n- \frac{1}{2}$, which is optimal in case of $(p_1,
p_2, p)= (1,1, \frac{1}{2})$. It is not the
case for the Heisenberg group because the product of two Heisenberg groups
is not a Heisenberg group. We will give a pointwise estimate for the kernel
of the bilinear Riesz means $S^{\alpha}$, which is similar to the known
estimate for the kernel of the Riesz means on the Heisenberg group but very
worse than the corresponding estimate for bilinear Bochner-Riesz operator $
B^{\alpha}$. Such a pointwise estimate only gives a very rough result for
smoothness index $\alpha (p_{1},p_{2})$. We don't know if there exist a
better pointwise estimate even for the Riesz means on the Heisenberg group
(cf. \cite{Mauc} or \cite{Thang}). So we have to develop a new technique to
obtain a better result about index $\alpha (p_{1},p_{2})$, for example, in
case of $(p_1, p_2, p)= (\infty, \infty, \infty)$. Secondly, on Euclidean
space, the Fourier transform of the product of two functions is the
convolution of Fourier transforms of two functions because the dual of the
Euclidean space is itself. As a consequence, the $L^{2} \times L^{2}
\rightarrow L^{2}$ boundedness holds for suitable bilinear Fourier
multiplier, which play an important role for giving the estimate in case of $
(p_1, p_2, p)= (2, \infty, 2)$. But this convenience false on the Heisenberg
group. Finally, the restriction estimate for the sublaplacian is very
different from that for the Laplacian on the Euclidean space because the
Heisenberg group has the center of dimension one. Fefferman \cite{Feff}
pointed out that the restriction theorem apply to the study of the
boundedness of Bochner-Riesz means. Stein's earlier result in \cite{Stein}
was improved by using the restriction theorem. Mauceri \cite{Mauc}
investigated the $L^{p}$ boundedness of Riesz means on the Heisenberg group.
Mauceri's result corresponds with Stein's earlier result. M\"{u}ller \cite
{Mull} construct a counter-example to show that the usual norm estimate of
restriction operators holds only in the trivial case $p=1$. M\"{u}ller \cite
{Mull2} gave a new proof of Mauceri's result by using a revised restriction
estimate but didn't improve Mauceri's result. As a result of above reasons,
our techniques are very different from that in \cite{Bern}.

This article is organized as follows. In the next section, we state some
basic facts about the Heisenberg group, and summarize our full results in
the main theorem. In Section 3, we give the pointwise estimate for the
kernel of $S^{\alpha }$. We prove the $L^{p_{1}}\times L^{p_{2}} \rightarrow
L^{p}$ boundedness of $S^{\alpha }$ in the case of $1 \leq p_{1}, p_{2} \leq
2$ in Section 4 and for some particular triples of points $(p_{1}, p_{2}, p)$
in Section 5. The boundedness in other cases are derived from the results in
Section 4 and Section 5 by using of bilinear interpolation method. We
outline this argument in Appendix for reader's convenience.

\vskip 0.5 cm

\section{Preliminaries}

First we recall some basic facts about the Heisenberg group. These facts are
familiar and easy to find in many references. Let $\mathbb{H}^{n}$ denote
the Heisenberg group whose underlying manifold is $\mathbb{C}^{n}\times
\mathbb{R}$ and the group law is given by
\begin{equation*}
(z,t)(w,s)=(z+w,t+s+\frac{1}{2} \text{Im}(z\cdot \overline{w})).
\end{equation*}
The Haar measure on $\mathbb{H}^{n}$ coincides with the Lebesgue measure on $
\mathbb{C}^{n}\times\mathbb{R}$. A homogeneous structure on $\mathbb{H}^{n}$
is given by the non-isotropic dilations $\delta_{r}(z,t)=(rz,r^{2}t)$. We
define a homogeneous norm on $\mathbb{H}^{n}$ by
\begin{equation*}
\left\vert x \right\vert = \Big( \frac{1}{16} \left\vert z\right\vert ^{4}+
t^{2} \Big)^{\frac{1}{4}}, \quad x =(z,t)\in \mathbb{H}^{n}.
\end{equation*}
This norm satisfies the triangle inequality and leads to a left-invariant
distance $d(x,y) = \left\vert x^{-1}y \right\vert$. The ball of radius $r$
centered at $x$ is
\begin{equation*}
B(x,r)= \{ y\in \mathbb{H}^{n}: \left\vert x^{-1}y \right\vert <r \}.
\end{equation*}
The Haar measure $dx$ satisfies $d\delta _{r}(x)=r^{Q}dx$ where $Q=2n+2$ is
the homogeneous dimension of $\mathbb{H}^{n}$. If $f$ and $g$ are functions
on $\mathbb{H}^{n}$, their convolution is defined by
\begin{equation*}
(f \ast g)(x)= \int_{\mathbb{H}^{n}} f\left( xy^{-1}\right)g(y)\, dy,\quad
x,y \in \mathbb{H}^{n}.
\end{equation*}
For each $\lambda \in \mathbb{R}^{\ast}$ and $f\in \mathscr{S}(\mathbb{H}^{n})$, the
inverse Fourier transform of $f$ in variable $t$ is defined by
\begin{equation*}
f^{\lambda }(z)=\int_{-\infty }^{\infty }e^{i\lambda t}f(z,t)\, dt.
\end{equation*}
An easy calculation shows that
\begin{equation*}
(f\ast g)^{\lambda }(z)=\int_{\mathbb{C}^{n}}f^{\lambda }(z-\omega
)g^{\lambda } (\omega )e^{\frac{i}{2}\lambda \mathrm{{Im}(z\cdot \overline{
\omega })}}\, d\omega ,\quad z,\omega \in \mathbb{C}^{n}.
\end{equation*}
Thus, we are led to the convolution of the form
\begin{equation*}
f\ast _{\lambda }g=\int_{\mathbb{C}^{n}}f(z-\omega )g(\omega ) e^{\frac{i}{2}
\lambda \mathrm{{Im}(z\cdot \overline{\omega })}}\, d\omega ,
\end{equation*}
which are called the $\lambda $-twisted convolution.

The sublaplacian $\mathcal{L}$ is defined by
\begin{equation*}
\mathcal{L}= -\sum_{j=1}^{n}(X_{j}^{2}+Y_{j}^{2})
\end{equation*}
where
\begin{eqnarray*}
X_{j} &=&\frac{\partial }{\partial x_{j}}+\frac{1}{2}y_{j}\frac{\partial }{
\partial t}, \quad j=1,2,\cdots ,n, \\
Y_{j} &=&\frac{\partial }{\partial y_{j}}-\frac{1}{2}x_{j}\frac{\partial }{
\partial t}, \quad j=1,2,\cdots ,n,
\end{eqnarray*}
are left invariant vector fields on $\mathbb{H}^{n}$. Up to a constant
multiple, $\mathcal{L}$ is the unique left invariant, rotation invariant
differential operator that is homogeneous of degree two. Therefore, it is
regarded as the counterpart of the Laplacian on $\mathbb{R}^{n}$. The
sublaplacian $\mathcal{L}$ is a positive and essentially self-adjoint
operator. In the following, we state the spectral decomposition of $\mathcal{L}$ (cf. \cite{Thang}).

Let $\varphi _{k}$ be the Laguerre functions on $\mathbb{C}^{n}$ given by
\begin{equation*}
\varphi _{k}(z)=L_{k}^{n-1}\big(\frac{1}{2}\left\vert z\right\vert ^{2}\big)
e^{-\frac{1}{4}\left\vert z\right\vert ^{2}},
\end{equation*}
where $L_{k}^{n-1}$ are the Laguerre polynomials of type $n-1$ defined on $
\mathbb{R}$ by
\begin{equation*}
L_{k}^{n-1}(t)e^{-t}t^{n-1}=\frac{1}{k!}\left( \frac{d}{dt}\right)
^{k}(e^{-t}t^{k+n-1}).
\end{equation*}
Define functions
\begin{equation*}
e_{k}^{\lambda }(z,t)=e^{-i\lambda t}\varphi _{k}^{\lambda }(z)=e^{-i\lambda
t}\varphi _{k}(\sqrt{\left\vert \lambda \right\vert }z),\quad \lambda \in
\mathbb{R}^{\ast }.
\end{equation*}
For $f\in L^{2}(\mathbb{H}^{n})$, we have the expansion
\begin{equation}
f(z,t)=\sum_{k=0}^{\infty }\int_{-\infty }^{\infty }f\ast e_{k}^{\lambda
}(z,t)\,d\mu (\lambda )  \label{expansion}
\end{equation}
where $d\mu (\lambda )=(2\pi )^{-n-1}\left\vert \lambda \right\vert
^{n}d\lambda $ is the Plancherel measure for $\mathbb{H}^{n}$. Each $f\ast
e_{k}^{\lambda }$ is the eigenfunction of $\mathcal{L}$ with eigenvalue $
(2k+n)\left\vert \lambda \right\vert $. We also have the Plancherel formula
\begin{equation}
\left\Vert f\right\Vert _{2}^{2}=(2\pi )^{-2n-1}\sum_{k=0}^{\infty
}\int_{-\infty }^{\infty }\int_{\mathbb{C}^{n}}\left\vert f^{\lambda }\ast
_{\lambda }\varphi _{k}^{\lambda }(z)\right\vert ^{2}\lambda
^{2n}\,dzd\lambda .  \label{Plancherel}
\end{equation}
Defining
\begin{equation*}
\widetilde{e}_{k}^{\lambda }(z,t)=e_{k}^{\frac{\lambda }{2k+n}}(z,t),
\end{equation*}
we can rewrite the decomposition (\ref{expansion}) as
\begin{equation*}
f(z,t)=\int_{-\infty }^{\infty }\sum_{k=0}^{\infty }(2k+n)^{-n-1}f\ast
\widetilde{e}_{k}^{\lambda }(z,t)\,d\mu (\lambda ).
\end{equation*}
Let
\begin{equation*}
P_{\lambda }f(z,t)=\sum_{k=0}^{\infty }(2k+n)^{-n-1}f\ast (\widetilde{e}
_{k}^{\lambda }+\widetilde{e}_{k}^{-\lambda })(z,t).
\end{equation*}
Then (\ref{expansion}) can be written as
\begin{equation*}
f(z,t)=\int_{0}^{\infty }P_{\lambda }f(z,t)\,d\mu (\lambda ).
\end{equation*}
It is clear that $P_{\lambda }f$ is an eigenfunction of the $\mathcal{L}$
with eigenvalue $\lambda $ and we have the spectral decomposition
\begin{equation*}
\mathcal{L}f=\int_{0}^{\infty }\lambda P_{\lambda }f\,d\mu (\lambda ).
\end{equation*}

Now we define the bilinear Riesz means associated to the sublaplacian $
\mathcal{L}$ for $f,g\in \mathscr{S}(\mathbb{H}^{n})$ by
\begin{equation*}
S_{R}^{\alpha }(f,g)=\int_{0}^{\infty }\int_{0}^{\infty }\left( 1-\frac{
\lambda _{1}+\lambda _{2}}{R}\right) _{+}^{\alpha }P_{\lambda
_{1}}fP_{\lambda _{2}}g\,d\mu (\lambda _{1})d\mu (\lambda _{2}).
\end{equation*}
It is easy to see that
\begin{equation*}
S_{R}^{\alpha }(f,g)(x)=\int_{\mathbb{H}^{n}}\int_{\mathbb{H}^{n}}f(x\omega
_{1}^{-1})g(x\omega _{2}^{-1})S_{R}^{\alpha }(\omega _{1},\omega
_{2})\,d\omega _{1}d\omega _{2},
\end{equation*}
where the kernel is given by
\begin{eqnarray*}
&&S_{R}^{\alpha }\big((z_{1},t_{1}),(z_{2},t_{2})\big)=\sum_{k=0}^{\infty
}\sum_{l=0}^{\infty }\int_{-\infty }^{\infty }\int_{-\infty }^{\infty
}\left( 1-\frac{(2k+n)\left\vert \lambda _{1}\right\vert +(2l+n)\left\vert
\lambda _{2}\right\vert }{R}\right) _{+}^{\alpha } \\
&&\qquad \qquad \qquad \qquad \qquad \qquad \qquad \qquad \qquad \times
e_{k}^{\lambda _{1}}(z_{1},t_{1})e_{l}^{\lambda _{2}}(z_{2},t_{2})\,d\mu
(\lambda _{1})d\mu (\lambda _{2}).
\end{eqnarray*}
Because
\begin{equation*}
S_{R}^{\alpha }\big((z_{1},t_{1}),(z_{2},t_{2})\big)=R^{Q}S_{1}^{\alpha }
\big((\sqrt{R}z_{1},Rt_{1}),(\sqrt{R}z_{2},Rt_{2})\big).
\end{equation*}
By a dilation argument, the $L^{p_{1}}\times L^{p_{2}}\rightarrow L^{p}$
boundedness of $S_{R}^{\alpha }$ is deduced from the $L^{p_{1}}\times
L^{p_{2}}\rightarrow L^{p}$ boundedness of $S_{1}^{\alpha }$ as $
1/p=1/p_{1}+1/p_{2}$. We will concentrate on the operator $S_{1}^{\alpha }$
and write $S^{\alpha }$ instead of $S_{1}^{\alpha }$.

Our full results are summarized in the following theorem.

\textbf{Main Theorem}. \textit{Let $1\leq p_{1},p_{2}\leq \infty $ and $
1/p=1/p_{1}+1/p_{2}$. }

\textit{(1) (region I) For $2\leq p_{1},p_{2}\leq \infty $ and $p\geq 2$, if
$\alpha >Q\left( 1-\frac{1}{p}\right) - \frac{1}{2}$, then $S^{\alpha }$ is
bounded from $L^{p_{1}}(\mathbb{H}^{n})\times L^{p_{2}} (\mathbb{H}^{n})$ to
$L^{p}(\mathbb{H}^{n})$. }

\textit{(2) (region II) For $2\leq p_{1},p_{2}\leq \infty $ and $1\leq p\leq
2$, if $\alpha >(Q-1)\left( 1-\frac{1}{p}\right) $, then $S^{\alpha }$ is
bounded from $L^{p_{1}}(\mathbb{H}^{n})\times L^{p_{2}}(\mathbb{H}^{n})$ to $
L^{p}(\mathbb{H}^{n})$. }

\textit{(3) (region III) For $1\leq p_{1}\leq 2\leq p_{2}\leq \infty $ and $
p\geq 1$, if $\alpha >Q\left( \frac{1}{2}-\frac{1}{p_{2}}\right) - \left( 1-
\frac{1}{p}\right)$, then $S^{\alpha }$ is bounded from $L^{p_{1}}(\mathbb{H}
^{n})\times L^{p_{2}}(\mathbb{H}^{n})$ to $L^{p}(\mathbb{H}^{n})$; For $
1\leq p_{2}\leq 2\leq p_{1}\leq \infty $ and $p\geq 1$, if $\alpha >Q\left(
\frac{1}{2}-\frac{1}{p_{1}}\right) - \left( 1-\frac{1}{p}\right)$, then $
S^{\alpha }$ is bounded from $L^{p_{1}} (\mathbb{H}^{n})\times L^{p_{2}}(
\mathbb{H}^{n})$ to $L^{p}(\mathbb{H}^{n})$. }

\textit{(4) (region IV) For $1\leq p_{1}\leq 2\leq p_{2}\leq \infty $ and $
p\leq 1$ , if $\alpha >Q\left( \frac{1}{p_{1}}-\frac{1}{2}\right) $, then $
S^{\alpha } $ is bounded from $L^{p_{1}}(\mathbb{H}^{n})\times L^{p_{2}}(
\mathbb{H} ^{n})$ to $L^{p}(\mathbb{H}^{n})$; For $1\leq p_{2}\leq 2\leq
p_{1}\leq \infty $ and $p\leq 1$, if $\alpha >Q\left( \frac{1}{p_{2}}-\frac{1
}{2} \right) $, then $S^{\alpha }$ is bounded from $L^{p_{1}}(\mathbb{H}
^{n})\times L^{p_{2}}(\mathbb{H}^{n})$ to $L^{p}(\mathbb{H}^{n})$. }

\textit{(5) (region V) For $1\leq p_{1},$ $p_{2}\leq 2$, if $\alpha >Q\left(
\frac{1 }{p}-1\right) $, then $S^{\alpha }$ is bounded from $L^{p_{1}}(
\mathbb{H} ^{n})\times L^{p_{2}}(\mathbb{H}^{n})$ to $L^{p}(\mathbb{H}^{n}).$
}

\includegraphics[scale=0.9]{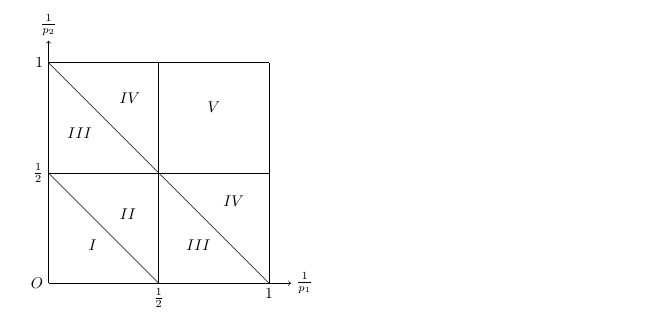}

\section{Pointwise estimate for the kernel}

Note that $S^{\alpha }\big((z_{1},t_{1}),(z_{2},t_{2})\big)$ is bi-radial
with respect to $z_{1}$ and $z_{2}$, which means $S^{\alpha }\big(
(z_{1},t_{1}),(z_{2},t_{2})\big)$ depends only on $\left\vert
z_{1}\right\vert ,\left\vert z_{2}\right\vert ,t_{1},t_{2}$. If $F\big(
(z_{1},t_{1}),(z_{2},t_{2})\big)$ is bi-radial with respect to $z_{1}$ and $
z_{2}$, we also write it as $F\big((r_{1},t_{1}),(r_{2},t_{2})\big)$ for
convenience where $r_{1}=\left\vert z_{1}\right\vert $, $r_{2}=\left\vert
z_{2}\right\vert $.

Suppose $F\big((z_{1},t_{1}),(z_{2},t_{2})\big)$ is bi-radial with respect
to $z_{1}$ and $z_{2}$, we define
\begin{eqnarray}
&&R_{k,l}(\lambda _{1},\lambda_{2},F)=\frac{2^{(1-n)}k!}{(k+n-1)!}\frac{
2^{(1-n)}l!}{(l+n-1)!}  \label{equmain} \\
&&\qquad \qquad \qquad \qquad \times \int_{0}^{\infty} \int_{0}^{\infty}
F^{\lambda_{1},\lambda _{2}}(r_{1},r_{2}) \varphi _{k}^{\lambda
_{1}}(r_{1})\varphi_{l}^{\lambda_{2}}(r_{2})r_{1}^{2n-1}r_{2}^{2n-1}\,
dr_{1}dr_{2},  \notag
\end{eqnarray}
where
\begin{equation*}
F^{\lambda _{1},\lambda _{2}}(r_{1},r_{2})=\int_{-\infty }^{\infty}
\int_{-\infty}^{\infty} e^{i\lambda_{1}t_{1}} e^{i\lambda_{2}t_{2}} F\big(
(r_{1},t_{1}),(r_{2},t_{2})\big)\, dt_{1}dt_{2}.
\end{equation*}
Taking use of the Laguerre transform, we have
\begin{eqnarray*}
F\big((r_{1},t_{1}),(r_{2},t_{2})\big)=\sum_{k=0}^{\infty }
\sum_{l=0}^{\infty}\int_{-\infty }^{\infty } \int_{-\infty}^{\infty
}e^{-i\left( \lambda_{1}t_{1}+\lambda_{2}t_{2}\right)}
R_{k,l}(\lambda_{1},\lambda_{2},F) \varphi_{k}^{\lambda_{1}}(r_{1}) \varphi
_{l}^{\lambda_{2}}(r_{2})\, d\mu (\lambda _{1})d\mu (\lambda _{2}).
\end{eqnarray*}
Especially, if
\begin{equation*}
\int_{-\infty }^{\infty }\int_{-\infty }^{\infty
}\left(\sum_{k=0}^{\infty}\sum_{l=0}^{\infty } \Big\vert R_{k,l}(\lambda
_{1},\lambda_{2},F)\Big\vert \frac{(k+n-1)!}{k!}\frac{(l+n-1)!}{l!}\right)\,
d\mu (\lambda _{1})d\mu (\lambda _{2})< \infty,
\end{equation*}
then $F$ is bounded, because
\begin{equation*}
\left\Vert \varphi _{k}\right\Vert _{\infty }= c_{n}\frac{(k+n-1)!}{k!}.
\end{equation*}

\begin{theorem}
\label{generalth}If $\alpha >4m-1$ where $m$ is a positive integer, then for
any $\omega _{1}=(z_{1},t_{1}), \omega _{2}=(z_{2},t_{2}) \in \mathbb{H}
^{n}. $
\begin{equation*}
\left\vert S^{\alpha }(\omega _{1},\omega _{2})\right\vert \leq
C(1+\left\vert \omega _{1}\right\vert )^{-2m}(1+\left\vert
\omega_{2}\right\vert )^{-2m}.
\end{equation*}
\end{theorem}

\begin{proof}
\quad Set $F\big((z_{1},t_{1}),(z_{2},t_{2})\big)=S^{\alpha }\big(
(z_{1},t_{1}),(z_{2},t_{2})\big)$. If we can show that
\begin{eqnarray}
&&\int_{-\infty }^{\infty }\int_{-\infty }^{\infty }\bigg(\sum_{k=0}^{\infty
}\sum_{l=0}^{\infty }\left\vert R_{k,l}\Big(\lambda _{1},\lambda
_{2},(it_{1}-\frac{1}{4}\left\vert z_{1}\right\vert ^{2})^{m}(it_{2}-\frac{1
}{4}\left\vert z_{2}\right\vert ^{2})^{m}F\Big)\right\vert  \label{equmain2}
\\
&&\qquad \qquad \qquad \qquad \times \frac{(k+n-1)!}{k!}\frac{(l+n-1)!}{l!}
\bigg)\,d\mu (\lambda _{1})d\mu (\lambda _{2})<\infty ,  \notag
\end{eqnarray}
then $(it_{1}-\frac{1}{4}\left\vert z_{1}\right\vert ^{2})^{m}(it_{2}-\frac{
1 }{4}\left\vert z_{2}\right\vert ^{2})^{m}S^{\alpha }\big(
(z_{1},t_{1}),(z_{2},t_{2})\big)$ is bounded. It follows that
\begin{equation*}
\left\vert S^{\alpha }(\omega _{1},\omega _{2})\right\vert \leq
C(1+\left\vert \omega _{1}\right\vert )^{-2m}(1+\left\vert \omega
_{2}\right\vert )^{-2m}
\end{equation*}
and Theorem \ref{generalth} is proved.

It is clear that $R_{k,l}(\lambda _{1},\lambda _{2},F)=\big(
1-(2k+n)\left\vert \lambda _{1}\right\vert -(2l+n)\left\vert \lambda
_{2}\right\vert \big)_{+}^{\alpha }$. We first calculate $R_{k,l}\big(
\lambda _{1},\lambda _{2},\big(it_{1}-\frac{1}{4}\left\vert z_{1}\right\vert
^{2}\big)\big(it_{2}-\frac{1}{4}\left\vert z_{2}\right\vert ^{2}\big)F\big)$
and assume that $\lambda _{1},\lambda _{2}>0$. From (\ref{equmain}), we have
\begin{eqnarray*}
&&R_{k,l}(\lambda _{1},\lambda _{2},it_{1}\cdot it_{2}F)=\frac{2^{(1-n)}k!}{
(k+n-1)!}\frac{2^{(1-n)}l!}{(l+n-1)!} \\
&&\qquad \qquad \qquad \times \int_{0}^{\infty }\int_{0}^{\infty }\left(
it_{1}\cdot it_{2}F\right) ^{\lambda _{1},\lambda _{2}}(r_{1},r_{2})\varphi
_{k}^{\lambda _{1}}(r_{1})\varphi _{l}^{\lambda
_{2}}(r_{2})r_{1}^{2n-1}r_{2}^{2n-1}\,dr_{1}dr_{2}.
\end{eqnarray*}
Since that
\begin{equation*}
\left( it_{1}\cdot it_{2}F\right) ^{\lambda _{1},\lambda _{2}}(r_{1},r_{2})=
\frac{\partial }{\partial \lambda _{1}}\frac{\partial }{\partial \lambda _{2}
}F^{\lambda _{1},\lambda _{2}}(r_{1},r_{2}),
\end{equation*}
we get
\begin{eqnarray*}
&&R_{k,l}(\lambda _{1},\lambda _{2},it_{1}\cdot it_{2}F) \\
&=&\frac{\partial }{\partial \lambda _{1}}\frac{\partial }{\partial \lambda
_{2}}R_{k,l}(\lambda _{1},\lambda _{2},F) \\
&&-\frac{2^{(1-n)}k!}{(k+n-1)!}\frac{2^{(1-n)}l!}{(l+n-1)!}\int_{0}^{\infty
}\int_{0}^{\infty }\frac{\partial }{\partial \lambda _{1}}F^{\lambda
_{1},\lambda _{2}}(r_{1},r_{2})\varphi _{k}^{\lambda _{1}}(r_{1})\left(
\frac{\partial }{\partial \lambda _{2}}\varphi _{l}^{\lambda
_{2}}(r_{2})\right) r_{1}^{2n-1}r_{2}^{2n-1}\,dr_{1}dr_{2} \\
&&-\frac{2^{(1-n)}k!}{(k+n-1)!}\frac{2^{(1-n)}l!}{(l+n-1)!}\int_{0}^{\infty
}\int_{0}^{\infty }\frac{\partial }{\partial \lambda _{2}}F^{\lambda
_{1},\lambda _{2}}(r_{1},r_{2})\left( \frac{\partial }{\partial \lambda _{1}}
\varphi _{k}^{\lambda _{1}}(r_{1})\right) \varphi _{l}^{\lambda
_{2}}(r_{2})r_{1}^{2n-1}r_{2}^{2n-1}\,dr_{1}dr_{2} \\
&&-\frac{2^{(1-n)}k!}{(k+n-1)!}\frac{2^{(1-n)}l!}{(l+n-1)!}\int_{0}^{\infty
}\int_{0}^{\infty }F^{\lambda _{1},\lambda _{2}}(r_{1},r_{2})\left( \frac{
\partial }{\partial \lambda _{1}}\varphi _{k}^{\lambda _{1}}(r_{1})\right)
\left( \frac{\partial }{\partial \lambda _{2}}\varphi _{l}^{\lambda
_{2}}(r_{2})\right) r_{1}^{2n-1}r_{2}^{2n-1}\,dr_{1}dr_{2}.
\end{eqnarray*}
Noticing the fact (see \cite{Thang}, p. 92)
\begin{equation*}
\frac{\partial }{\partial \lambda }\varphi _{k}^{\lambda }(r)=\frac{k}{
\lambda }\varphi _{k}^{\lambda }(r)-\frac{(k+n-1)}{\lambda }\varphi
_{k-1}^{\lambda }(r)-\frac{1}{4}r^{2}\varphi _{k}^{\lambda }(r),
\end{equation*}
it then follows that
\begin{eqnarray*}
&&R_{k,l}(\lambda _{1},\lambda _{2},it_{1}\cdot it_{2}F) \\
&=&\frac{\partial }{\partial \lambda _{1}}\frac{\partial }{\partial \lambda
_{2}}R_{k,l}(\lambda _{1},\lambda _{2},F)-R_{k,l}\left( \lambda _{1},\lambda
_{2},\frac{1}{4}\left\vert z_{1}\right\vert ^{2}\left\vert z_{2}\right\vert
^{2}\right) \\
&&+\frac{l}{\lambda _{2}}\frac{\partial }{\partial \lambda _{1}}\left(
R_{k,l-1}(\lambda _{1},\lambda _{2},F)-R_{k,l}(\lambda _{1},\lambda
_{2},F)\right) \\
&&+\frac{k}{\lambda _{1}}\frac{\partial }{\partial \lambda _{2}}\left(
R_{k-1,l}(\lambda _{1},\lambda _{2},F)-R_{k,l}(\lambda _{1},\lambda
_{2},F)\right) \\
&&+\frac{k}{\lambda _{1}}\frac{l}{\lambda _{2}}\left( R_{k-1,l}(\lambda
_{1},\lambda _{2},F)+R_{k,l-1}(\lambda _{1},\lambda _{2},F)-R_{k,l}(\lambda
_{1},\lambda _{2},F)-R_{k-1,l-1}(\lambda _{1},\lambda _{2},F)\right) \\
&&+\frac{\partial }{\partial \lambda _{1}}R_{k,l}(\lambda _{1},\lambda _{2},
\frac{1}{4}\left\vert z_{2}\right\vert ^{2}F)+\frac{\partial }{\partial
\lambda _{2}}R_{k,l}(\lambda _{1},\lambda _{2},\frac{1}{4}\left\vert
z_{1}\right\vert ^{2}F) \\
&&+\frac{k}{\lambda _{1}}\left( R_{k,l}(\lambda _{1},\lambda _{2},\frac{1}{4}
\left\vert z_{2}\right\vert ^{2}F)-R_{k-1,l}(\lambda _{1},\lambda _{2},\frac{
1}{4}\left\vert z_{2}\right\vert ^{2}F)\right) \\
&&+\frac{l}{\lambda _{2}}\left( R_{k,l}(\lambda _{1},\lambda _{2},\frac{1}{4}
\left\vert z_{1}\right\vert ^{2}F)-R_{k,l-1}(\lambda _{1},\lambda _{2},\frac{
1}{4}\left\vert z_{1}\right\vert ^{2}F)\right) .
\end{eqnarray*}
Similarly, since that
\begin{eqnarray*}
\left( -it_{1}\frac{1}{4}\left\vert z_{2}\right\vert ^{2}F\right) ^{\lambda
_{1},\lambda _{2}}(r_{1},r_{2}) &=&-\frac{\partial }{\partial \lambda _{1}}
\left( \frac{1}{4}\left\vert z_{2}\right\vert ^{2}F\right) ^{\lambda
_{1},\lambda _{2}}(r_{1},r_{2}), \\
\left( -it_{2}\frac{1}{4}\left\vert z_{1}\right\vert ^{2}F\right) ^{\lambda
_{1},\lambda _{2}}(r_{1},r_{2}) &=&-\frac{\partial }{\partial \lambda _{2}}
\left( \frac{1}{4}\left\vert z_{1}\right\vert ^{2}F\right) ^{\lambda
_{1},\lambda _{2}}(r_{1},r_{2}),
\end{eqnarray*}
we have
\begin{eqnarray*}
R_{k,l}(\lambda _{1},\lambda _{2},-it_{1}\frac{1}{4}\left\vert
z_{2}\right\vert ^{2}F) &=&-\frac{\partial }{\partial \lambda _{1}}
R_{k,l}(\lambda _{1},\lambda _{2},\frac{1}{4}\left\vert z_{2}\right\vert
^{2}F)-R_{k,l}\left( \lambda _{1},\lambda _{2},\frac{1}{4}\left\vert
z_{1}\right\vert ^{2}\left\vert z_{2}\right\vert ^{2}\right) \\
&&+\frac{k}{\lambda _{1}}\left( R_{k,l}(\lambda _{1},\lambda _{2},\frac{1}{4}
\left\vert z_{2}\right\vert ^{2}F)-R_{k-1,l}(\lambda _{1},\lambda _{2},\frac{
1}{4}\left\vert z_{2}\right\vert ^{2}F)\right) ,
\end{eqnarray*}
and
\begin{eqnarray*}
R_{k,l}(\lambda _{1},\lambda _{2},-it_{2}\frac{1}{4}\left\vert
z_{1}\right\vert ^{2}F) &=&-\frac{\partial }{\partial \lambda _{2}}
R_{k,l}(\lambda _{1},\lambda _{2},\frac{1}{4}\left\vert z_{1}\right\vert
^{2}F)-R_{k,l}\left( \lambda _{1},\lambda _{2},\frac{1}{4}\left\vert
z_{1}\right\vert ^{2}\left\vert z_{2}\right\vert ^{2}\right) \\
&&+\frac{l}{\lambda _{2}}\left( R_{k,l}(\lambda _{1},\lambda _{2},\frac{1}{4}
\left\vert z_{1}\right\vert ^{2}F)-R_{k,l-1}(\lambda _{1},\lambda _{2},\frac{
1}{4}\left\vert z_{1}\right\vert ^{2}F)\right) .
\end{eqnarray*}
Therefore,
\begin{eqnarray*}
&&R_{k,l}(\lambda _{1},\lambda _{2},\left( it_{1}-\frac{1}{4}\left\vert
z_{1}\right\vert ^{2}\right) \left( it_{2}-\frac{1}{4}\left\vert
z_{2}\right\vert ^{2}\right) F) \\
&=&\frac{\partial }{\partial \lambda _{1}}\frac{\partial }{\partial \lambda
_{2}}R_{k,l}(\lambda _{1},\lambda _{2},F)-2R_{k,l}\left( \lambda
_{1},\lambda _{2},\frac{1}{4}\left\vert z_{1}\right\vert ^{2}\frac{1}{4}
\left\vert z_{2}\right\vert ^{2}\right) \\
&&+\frac{l}{\lambda _{2}}\frac{\partial }{\partial \lambda _{1}}\left(
R_{k,l-1}(\lambda _{1},\lambda _{2},F)-R_{k,l}(\lambda _{1},\lambda
_{2},F)\right) \\
&&+\frac{k}{\lambda _{1}}\frac{\partial }{\partial \lambda _{2}}\left(
R_{k-1,l}(\lambda _{1},\lambda _{2},F)-R_{k,l}(\lambda _{1},\lambda
_{2},F)\right) \\
&&+\frac{k}{\lambda _{1}}\frac{l}{\lambda _{2}}\left( R_{k-1,l}(\lambda
_{1},\lambda _{2},F)+R_{k,l-1}(\lambda _{1},\lambda _{2},F)-R_{k,l}(\lambda
_{1},\lambda _{2},F)-R_{k-1,l-1}(\lambda _{1},\lambda _{2},F)\right) \\
&&+2\frac{k}{\lambda _{1}}\left( R_{k,l}(\lambda _{1},\lambda _{2},\frac{1}{4
}\left\vert z_{2}\right\vert ^{2}F)-R_{k-1,l}(\lambda _{1},\lambda _{2},
\frac{1}{4}\left\vert z_{2}\right\vert ^{2}F)\right) \\
&&+2\frac{l}{\lambda _{2}}\left( R_{k,l}(\lambda _{1},\lambda _{2},\frac{1}{4
}\left\vert z_{1}\right\vert ^{2}F)-R_{k,l-1}(\lambda _{1},\lambda _{2},
\frac{1}{4}\left\vert z_{1}\right\vert ^{2}F)\right) .
\end{eqnarray*}

Let $\sigma =(2k+n)\lambda _{1}+(2l+n)\lambda _{2}$ so that $R_{k,l}(\lambda
_{1},\lambda _{2},F)=\psi (\sigma )$ where
\begin{equation*}
\psi (\sigma )=(1-\sigma )_{+}^{\alpha }.
\end{equation*}
Define
\begin{equation*}
\psi _{1}(\sigma )=\left( it_{1}-\frac{1}{4}\left\vert z_{1}\right\vert
^{2}\right) \left( it_{2}-\frac{1}{4}\left\vert z_{2}\right\vert ^{2}\right)
F.
\end{equation*}
It can be written as
\begin{eqnarray*}
\psi _{1}(\sigma ) &=&\frac{\partial }{\partial \lambda _{1}}\frac{\partial
}{\partial \lambda _{2}}\psi (\sigma )-\frac{l}{\lambda _{2}}\frac{\partial
}{\partial \lambda _{1}}\left( \psi (\sigma )-\psi (\sigma -2\lambda
_{2})\right) -\frac{k}{\lambda _{1}}\frac{\partial }{\partial \lambda _{2}}
\left( \psi (\sigma )-\psi (\sigma -2\lambda _{1})\right) \\
&&-\frac{kl}{\lambda _{1}\lambda _{2}}\left( \psi (\sigma )-\psi (\sigma
-2\lambda _{2})-\psi (\sigma -2\lambda _{1})+\psi (\sigma -2\lambda
_{1}-2\lambda _{2})\right) \\
&&+C_{1}\frac{k}{\lambda _{1}}\left( \psi (\sigma )-\psi (\sigma -2\lambda
_{1}\right) +C_{2}\frac{l}{\lambda _{2}}\left( \psi (\sigma )-\psi (\sigma
-2\lambda _{2})\right) -C_{3}\psi (\sigma ).
\end{eqnarray*}
The function $\psi (\sigma )$ has two properties: the first one is that $
\psi (\sigma )$ is supported in a set of the form $0<(2k+n)\lambda
_{1}+(2l+n)\lambda _{2}\leq c$ for large $k$ and $l$, and the second is
\begin{equation}
\int_{0}^{\infty }\int_{0}^{\infty }\left\vert \psi \left( (2k+n)\lambda
_{1}+(2l+n)\lambda _{2}\right) \right\vert \,d\mu (\lambda _{1})d\mu
(\lambda _{2})\leq C(2k+n)^{-n-1}(2l+n)^{-n-1}  \label{k2}
\end{equation}
We claim that the function $\psi _{1}(\sigma )$ also satisfies the same
conditions. The first property is obvious. To verify that $\psi _{1}(\sigma
) $ satisfies the second property (\ref{k2}), we rewrite $\psi _{1}(\sigma )$
as
\begin{eqnarray*}
\psi _{1}(\sigma ) &=&\frac{nk}{2\lambda _{1}\lambda _{2}}\frac{\partial }{
\partial l}\frac{\partial }{\partial k}\psi ((2k+n)\lambda
_{1}+(2l+n)\lambda _{2}) \\
&&-\frac{nk}{2\lambda _{1}\lambda _{2}}\frac{\partial }{\partial l}\left(
\psi ((2k+n)\lambda _{1}+(2l+n)\lambda _{2})-\psi ((2k-2+n)\lambda
_{1}+(2l+n)\lambda _{2})\right) \\
&&+\frac{nl}{2\lambda _{1}\lambda _{2}}\frac{\partial }{\partial k}\frac{
\partial }{\partial l}\psi ((2k+n)\lambda _{1}+(2l+n)\lambda _{2}) \\
&&-\frac{nl}{2\lambda _{1}\lambda _{2}}\frac{\partial }{\partial k}\left(
\psi ((2k+n)\lambda _{1}+(2l+n)\lambda _{2})-\psi ((2k+n)\lambda
_{1}+(2l-2+n)\lambda _{2})\right) \\
&&-\frac{kl}{\lambda _{1}\lambda _{2}}\frac{\partial }{\partial k}\frac{
\partial }{\partial l}\psi ((2k+n)\lambda _{1}+(2l+n)\lambda _{2}) \\
&&+\frac{kl}{\lambda _{1}\lambda _{2}}\frac{\partial }{\partial k}\left(
\psi ((2k+n)\lambda _{1}+(2l+n)\lambda _{2})-\psi ((2k+n)\lambda
_{1}+(2l-2+n)\lambda _{2})\right) \\
&&+\frac{kl}{\lambda _{1}\lambda _{2}}\frac{\partial }{\partial l}\psi
((2k+n)\lambda _{1}+(2l+n)\lambda _{2}) \\
&&-\frac{kl}{\lambda _{1}\lambda _{2}}\left( \psi ((2k+n)\lambda
_{1}+(2l+n)\lambda _{2})+\psi ((2k+n)\lambda _{1}+(2l-2+n)\lambda
_{2})\right) \\
&&-\frac{kl}{\lambda _{1}\lambda _{2}}\frac{\partial }{\partial l}\psi
((2k-2+n)\lambda _{1}+(2l+n)\lambda _{2}) \\
&&+\frac{kl}{\lambda _{1}\lambda _{2}}\left( \psi ((2k-2+n)\lambda
_{1}+(2l+n)\lambda _{2})-\psi ((2k-2+n)\lambda _{1}+(2l-2+n)\lambda
_{2})\right) \\
&&+2\frac{kl}{\lambda _{1}\lambda _{2}}\frac{\partial }{\partial l}\frac{
\partial }{\partial k}\psi ((2k+n)\lambda _{1}+(2l+n)\lambda _{2}) \\
&&-2\frac{kl}{\lambda _{1}\lambda _{2}}\frac{\partial }{\partial l}\left(
\psi ((2k+n)\lambda _{1}+(2l+n)\lambda _{2})-\psi ((2k-2+n)\lambda
_{1}+(2l+n)\lambda _{2})\right) \\
&&+2\frac{kl}{\lambda _{1}\lambda _{2}}\frac{\partial }{\partial k}\frac{
\partial }{\partial l}\psi ((2k+n)\lambda _{1}+(2l+n)\lambda _{2}) \\
&&-2\frac{kl}{\lambda _{1}\lambda _{2}}\frac{\partial }{\partial k}\left(
\psi ((2k+n)\lambda _{1}+(2l+n)\lambda _{2})-\psi ((2k+n)\lambda
_{1}+(2l-2+n)\lambda _{2})\right) \\
&&-C_{1}\frac{k}{\lambda _{1}}\frac{\partial }{\partial k}\psi
((2k+n)\lambda _{1}+(2l+n)\lambda _{2}) \\
&&+C_{1}\frac{k}{\lambda _{1}}\left( \psi ((2k+n)\lambda _{1}+(2l+n)\lambda
_{2})-\psi ((2k-2+n)\lambda _{1}+(2l+n)\lambda _{2})\right) \\
&&-C_{2}\frac{l}{\lambda _{2}}\frac{\partial }{\partial l}\psi
((2k+n)\lambda _{1}+(2l+n)\lambda _{2}) \\
&&+C_{2}\frac{l}{\lambda _{2}}\left( \psi ((2k+n)\lambda _{1}+(2l+n)\lambda
_{2})-\psi ((2k+n)\lambda _{1}+(2l-2+n)\lambda _{2})\right) \\
&&+\frac{n^{2}-2kl}{4\lambda _{1}\lambda _{2}}\frac{\partial }{\partial k}
\frac{\partial }{\partial l}\psi ((2k+n)\lambda _{1}+(2l+n)\lambda _{2}) \\
&&+C_{1}\frac{k}{\lambda _{1}}\frac{\partial }{\partial k}\psi
((2k+n)\lambda _{1}+(2l+n)\lambda _{2})+C_{2}\frac{l}{\lambda _{2}}\frac{
\partial }{\partial l}\psi ((2k+n)\lambda _{1}+(2l+n)\lambda _{2}) \\
&&-C_{0}\psi ((2k+n)\lambda _{1}+(2l+n)\lambda _{2}).
\end{eqnarray*}
Using Taylor expansion, we have that
\begin{eqnarray*}
\psi _{1}(\sigma ) &=&4nk\lambda _{1}\int_{k-1}^{k}(t+1-k)\psi
^{(3)}((2t+n)\lambda _{1}+(2l+n)\lambda _{2})dt \\
&&+4nl\lambda _{2}\int_{l-1}^{l}(s+1-l)\psi ^{(3)}((2k+n)\lambda
_{1}+(2s+n)\lambda _{2})ds \\
&&-16kl\lambda _{1}\lambda
_{2}\int_{l-1}^{l}\int_{k-1}^{k}(s+1-l)(t+1-k)\psi ^{(4)}((2t+n)\lambda
_{1}+(2s+n)\lambda _{2})dtds \\
&&+16kl\lambda _{1}\int_{k-1}^{k}(t+1-k)\psi ^{(3)}((2t+n)\lambda
_{1}+(2l+n)\lambda _{2}dt \\
&&+16kl\lambda _{2}\int_{l-1}^{l}(s+1-l)\psi ^{(3)}((2k+n)\lambda
_{1}+(2s+n)\lambda _{2})ds \\
&&+C_{1}k\lambda _{1}\int_{k-1}^{k}(t+1-k)\psi ^{(2)}((2t+n)\lambda
_{1}+(2l+n)\lambda _{2})dt \\
&&+C_{2}l\lambda _{2}\int_{l-1}^{l}(s+1-l)\psi ^{(2)}((2k+n)\lambda
_{1}+(2s+n)\lambda _{2})ds \\
&&+(n^{2}-8kl)\psi ^{(2)}((2k+n)\lambda _{1}+(2l+n)\lambda _{2}) \\
&&+C_{1}k\psi ^{(1)}((2k+n)\lambda _{1}+(2l+n)\lambda _{2})+C_{2}l\psi
^{(1)}((2k+n)\lambda _{1}+(2l+n)\lambda _{2}) \\
&&-C_{0}\psi ((2k+n)\lambda _{1}+(2l+n)\lambda _{2}).
\end{eqnarray*}
It follows that
\begin{eqnarray*}
&&\int_{0}^{\infty }\int_{0}^{\infty }\left\vert \psi _{1}(\sigma
)\right\vert \,d\mu (\lambda _{1})d\mu (\lambda _{2}) \\
&\leq &c_{0}\int_{0}^{\infty }\int_{0}^{\infty }\big(1-(2k+n)\lambda
_{1}-(2l+n)\lambda _{2}\big)_{+}^{\alpha }\,d\mu (\lambda _{1})d\mu (\lambda
_{2}) \\
&&+c_{1}\int_{0}^{\infty }\int_{0}^{\infty }\big(1-(2k+n)\lambda
_{1}-(2l+n)\lambda _{2}\big)_{+}^{\alpha -1}\,d\mu (\lambda _{1})d\mu
(\lambda _{2}) \\
&&+c_{2}\int_{0}^{\infty }\int_{0}^{\infty }\big(1-(2k+n)\lambda
_{1}-(2l+n)\lambda _{2}\big)_{+}^{\alpha -2}\,d\mu (\lambda _{1})d\mu
(\lambda _{2}) \\
&&+c_{3}\int_{k-1}^{k}\int_{0}^{\infty }\int_{0}^{\infty }\lambda _{1}\big(
1-(2t+n)\lambda _{1}-(2l+n)\lambda _{2}\big)_{+}^{\alpha -2}\,d\mu (\lambda
_{1})d\mu (\lambda _{2})dt \\
&&+c_{4}\int_{l-1}^{l}\int_{0}^{\infty }\int_{0}^{\infty }\lambda _{2}\big(
1-(2k+n)\lambda _{1}-(2s+n)\lambda _{2}\big)_{+}^{\alpha -2}\,d\mu (\lambda
_{1})d\mu (\lambda _{2})ds \\
&&+c_{5}\int_{k-1}^{k}\int_{0}^{\infty }\int_{0}^{\infty }\lambda _{1}\big(
1-(2t+n)\lambda _{1}-(2l+n)\lambda _{2}\big)_{+}^{\alpha -3}\,d\mu (\lambda
_{1})d\mu (\lambda _{2})dt \\
&&+c_{6}\int_{l-1}^{l}\int_{0}^{\infty }\int_{0}^{\infty }\lambda _{2}\big(
1-(2k+n)\lambda _{1}-(2s+n)\lambda _{2}\big)_{+}^{\alpha -3}\,d\mu (\lambda
_{1})d\mu (\lambda _{2})ds \\
&&+c_{7}\int_{k-1}^{k}\int_{l-1}^{l}\int_{0}^{\infty }\int_{0}^{\infty
}\lambda _{1}\lambda _{2}\big(1-(2t+n)\lambda _{1}-(2s+n)\lambda _{2}\big)
_{+}^{\alpha -4}\,d\mu (\lambda _{1})d\mu (\lambda _{2})dtds \\
&\leq &C(2k+n)^{-n-1}(2l+n)^{-n-1}.
\end{eqnarray*}
This proves that $\psi _{1}(\sigma )$ also has the second property (\ref{k2}). An iteration of the process shows that
\begin{equation*}
R_{k,l}\Big(\lambda _{1},\lambda _{2},(it_{1}-\frac{1}{4}
r_{1}^{2})^{j}(it_{2}-\frac{1}{4}r_{2}^{2})^{j}F\Big)=\psi _{j}(\sigma )
\end{equation*}
satisfies the condition
\begin{equation*}
\int_{0}^{\infty }\int_{0}^{\infty }\Big\vert\psi _{j}((2k+n)\lambda
_{1}+(2l+n)\lambda _{2})\Big\vert\,d\mu (\lambda _{1})d\mu (\lambda
_{2})\leq C_{j}(2k+n)^{-n-1}(2l+n)^{-n-1}
\end{equation*}
provided $(1-\lambda _{1}-\lambda _{2})_{+}^{\alpha -4}$ is integrable.
Thus, when $j=m$ and $\alpha >4m-1$, we have
\begin{eqnarray*}
&&\int_{0}^{\infty }\int_{0}^{\infty }\left\vert R_{k,l}\Big(\lambda
_{1},\lambda _{2},(it_{1}-\frac{1}{4}\left\vert z_{1}\right\vert
^{2})^{m}(it_{2}-\frac{1}{4}\left\vert z_{2}\right\vert ^{2})^{m}F\Big)
\right\vert \,d\mu (\lambda _{1})d\mu (\lambda _{2}) \\
&\leq &C_{m}(2k+n)^{-n-1}(2l+n)^{-n-1},
\end{eqnarray*}
which implies

\begin{eqnarray*}
&&\int_{0}^{\infty }\int_{0}^{\infty }\bigg(\sum_{k=0}^{\infty
}\sum_{l=0}^{\infty }\left\vert R_{k,l}\left( \lambda _{1},\lambda
_{2},(it_{1}-\frac{1}{4}\left\vert z_{1}\right\vert ^{2})^{m}(it_{2}-\frac{1
}{4}\left\vert z_{2}\right\vert ^{2})^{m}F\right) \right\vert \\
&&\qquad \qquad \qquad \qquad \times \frac{(k+n-1)!}{k!}\frac{(l+n-1)!}{l!}
\bigg)\,d\mu (\lambda _{1})d\mu (\lambda _{2})<\infty .
\end{eqnarray*}
We can get the same result by a similar calculation when $(\lambda
_{1},\lambda _{2})$ belongs to other quadrants. Then (\ref{equmain2}) holds
and the proof of Theorem \ref{generalth} is completed.
\end{proof}

As a consequence of Theorem \ref{generalth}, we have

\begin{corollary}
\label{cor}Let $1\leq p_{1},p_{2}\leq \infty $ and $1/p=1/p_{1}+1/p_{2}$. If
$\alpha >2Q+3$, then $S^{\alpha }$ is bounded from $L^{p_{1}}(\mathbb{H}
^{n})\times L^{p_{2}}(\mathbb{H}^{n})$ into $L^{p}(\mathbb{H}^{n})$.
\end{corollary}

\begin{proof}
\quad We take $m=\frac{Q}{2}+1$. If $\alpha >2Q+3$, we have $\alpha >4m-1$.
Then, Theorem \ref{generalth} is available. By H\"{o}lder's inequality and
Young's inequality, we conclude that
\begin{eqnarray*}
\left\Vert S^{\alpha }(f,g)\right\Vert _{p} &\leq &C\left\Vert f\right\Vert
_{p_{1}}\left\Vert g\right\Vert _{p_{2}}\int_{\mathbb{H}^{n}}(1+\left\vert
\omega _{1}\right\vert )^{-2m}d\omega _{1}\int_{\mathbb{H}^{n}}(1+\left\vert
\omega _{2}\right\vert )^{-2m}d\omega _{2} \\
&\leq &C\left\Vert f\right\Vert _{p_{1}}\left\Vert g\right\Vert
_{p_{2}}\left( \int_{1}^{\infty }t^{-2m+Q-1}dt\right) ^{2} \\
&\leq &C\left\Vert f\right\Vert _{p_{1}}\left\Vert g\right\Vert _{p_{2}}.
\end{eqnarray*}
The proof is completed.
\end{proof}

We note that the index in Corollary \ref{cor} is very high. In the rest part
of this paper, we will distinguish various cases and manage to give lower
indices .

\section{Boundedness of $S^{\protect\alpha }$ for $1\leq p_{1},p_{2}\leq 2$}

In this Section, we investigate the $L^{p_{1}}\times L^{p_{2}}\rightarrow
L^{p} $ boundedness of the bilinear Riesz means $S^{\alpha }$ for $1\leq
p_{1},p_{2}\leq 2$.

\begin{lemma}
\label{restriction} Suppose $m\in L^{\infty }(\mathbb{R})$. Define operator $
T_{m}f=\int_{a}^{b}m(\lambda )P_{\lambda }f\,d\mu (\lambda )$ for $0\leq a<b$
. Then, for any $1\leq p\leq 2$, we have
\begin{equation*}
\left\Vert T_{m}f\right\Vert _{2}\leq C\left\Vert m\right\Vert _{\infty
}\left( (b-a)b^{n}\right) ^{(\frac{1}{p}-\frac{1}{2})}\left\Vert
f\right\Vert _{p}.
\end{equation*}
\end{lemma}

\begin{proof}
\quad Since that
\begin{equation*}
P_{\lambda }f=\sum_{k=0}^{\infty }(2k+n)^{-n-1}f\ast (\widetilde{e}
_{k}^{\lambda }+\widetilde{e}_{k}^{-\lambda })(z,t)\text{, }
\end{equation*}
the operator $T_{m}$ can be written as
\begin{equation*}
T_{m}f=\int_{a}^{b}m(\lambda )P_{\lambda }f\,d\mu (\lambda )=f\ast G_{m},
\end{equation*}
where the kernel is given by
\begin{eqnarray*}
G_{m}(z,t) &=&\int_{a}^{b}m(\lambda )\sum_{k=0}^{\infty }(2k+n)^{-n-1}(
\widetilde{e}_{k}^{\lambda }+\widetilde{e}_{k}^{-\lambda })\,d\mu (\lambda )
\\
&=&C\sum_{k=0}^{\infty }\int_{\frac{a}{2k+n}}^{\frac{b}{2k+n}}(e^{-i\lambda
t}+e^{i\lambda t})m((2k+n)\lambda )\varphi _{k}^{\lambda
}(z)\vert\lambda\vert ^{n}d\lambda .
\end{eqnarray*}
Applying the Plancherel theorem in the variable $t$, the orthogonality of $
\varphi _{k}^{\lambda }$ and the fact
\begin{equation}
\left\Vert \varphi _{k}^{\lambda }\right\Vert _{2}\leq C \left\vert \lambda
\right\vert^{-\frac{n}{2 }}k^{\frac{n-1}{2}},  \label{phi2}
\end{equation}
we get that
\begin{eqnarray*}
\left\Vert G_{m}\right\Vert _{2}^{2} &\leq &C\int_{\mathbb{C} ^{n}}\int_{
\mathbb{R}}\left\vert \sum_{k=0}^{\infty }\int_{\frac{a}{2k+n}}^{\frac{b}{
2k+n} }(e^{-i\lambda t}+e^{i\lambda t})m((2k+n)\lambda )\varphi
_{k}^{\lambda }(z)\vert\lambda\vert ^{n}d\lambda \right\vert ^{2}\, dtdz \\
&=&C\int_{\mathbb{R}}\int_{\mathbb{C} ^{n}}\left\vert \sum_{k=0}^{\infty
}\chi _{[\frac{a}{2k+n},\frac{b}{ 2k+n}]}( \left\vert \lambda \right\vert
)m((2k+n)\left\vert \lambda \right\vert ) \varphi _{k}^{\lambda
}(z)\left\vert \lambda \right\vert ^{n}\right\vert ^{2}\, dzd\lambda \\
&\leq &C\int_{\mathbb{R}}\sum_{k=0}^{\infty }\left\vert \chi _{[\frac{a}{2k+n
},\frac{b}{2k+n} ]}(\left\vert \lambda \right\vert )m((2k+n)\left\vert
\lambda \right\vert )\right\vert ^{2} \left\Vert \varphi_{k}^{\lambda
}\right\Vert ^{2}_{2}\vert\lambda \vert^{2n}d\lambda \\
&\leq &C\left( \sum_{k=0}^{\infty }(2k+n)^{-n-1}k^{n-1}\right)
\int_{a}^{b}\left\vert m(\lambda )\right\vert ^{2}\vert\lambda\vert
^{n}d\lambda \\
&\leq &C\left\Vert m\right\Vert _{\infty }^{2}(b-a)b^{n}.
\end{eqnarray*}
This implies that
\begin{equation*}
\left\Vert T_{m}f\right\Vert _{2}\leq C\left( (b-a)b^{n}\right) ^{\frac{1}{2}
}\left\Vert m\right\Vert _{\infty }\left\Vert f\right\Vert _{1}\text{.}
\end{equation*}
By interpolation with the trivial estimate
\begin{equation*}
\left\Vert T_{m}f\right\Vert _{2}\leq \left\Vert m\right\Vert _{\infty
}\left\Vert f\right\Vert _{2},
\end{equation*}
we conclude that
\begin{equation*}
\left\Vert T_{m}f\right\Vert _{2}\leq C\left( (b-a)b^{n}\right) ^{(\frac{1}{
p }-\frac{1}{2})}\left\Vert m\right\Vert _{\infty }\left\Vert f\right\Vert
_{p}.
\end{equation*}
The proof is completed.
\end{proof}

\begin{theorem}
\label{mainTh} Suppose that $1\leq p_{1},p_{2}\leq 2$ and $
1/p=1/p_{1}+1/p_{2}$. If $\alpha >Q\left( \frac{1}{p}-1\right) $, then $
S^{\alpha }$ is bounded from $L^{p_{1}}(\mathbb{H}^{n})\times L^{p_{2}}(
\mathbb{H} ^{n}) $ into $L^{p}(\mathbb{H}^{n})$.
\end{theorem}

\begin{proof}
\quad We choose a nonnegative function $\varphi $ $\in C_{0}^{\infty }(\frac{
1}{2},2)$ satisfying $\sum_{-\infty }^{\infty }\varphi (2^{j}s)=1$, $s>0$.
For each $j\geq 0$, we set function
\begin{equation*}
\varphi _{j}^{\alpha }\left( s,t\right) =(1-s-t)_{+}^{\alpha }\varphi \big(
2^{j}\left( 1-s-t\right) \big),
\end{equation*}
and define bilinear operator
\begin{equation*}
T_{j}^{\alpha }(f,g)=\int_{0}^{\infty }\int_{0}^{\infty }\varphi
_{j}^{\alpha }\left( \lambda _{1},\lambda _{2}\right) P_{\lambda
_{1}}fP_{\lambda _{2}}g\,d\mu (\lambda _{1})d\mu (\lambda _{2}).
\end{equation*}
It is obvious that
\begin{equation*}
S^{\alpha }=\sum_{j=0}^{\infty }T_{j}^{\alpha },
\end{equation*}
and our result would follow if we can show that when $\alpha >Q\left( \frac{
1 }{p}-1\right) $, there exists an $\varepsilon >0$ such that for each $
j\geq 0 $,
\begin{equation}
\left\Vert T_{j}^{\alpha }\right\Vert _{L^{1}\times L^{p_{2}}\rightarrow
L^{p}}\leq C2^{-\varepsilon j}.  \label{Tj-}
\end{equation}
Fixing $j\geq 0$. In order to prove (\ref{Tj-}), we define $B_{j}=\{\omega
:\left\vert \omega \right\vert \leq 2^{j(1+\gamma )}\}\subseteq \mathbb{H}
^{n}$ and split the kernel $K_{j}^{\alpha }$ of $T_{j}^{\alpha }$ into four
parts:
\begin{equation*}
K_{j}^{\alpha }=K_{j}^{1}+K_{j}^{2}+K_{j}^{3}+K_{j}^{4},
\end{equation*}
where
\begin{eqnarray*}
K_{j}^{1}(\omega _{1},\omega _{2}) &=&K_{j}^{\alpha }(\omega _{1},\omega
_{2})\chi _{B_{j}}(\omega _{1})\chi _{B_{j}}(\omega _{2}), \\
K_{j}^{2}(\omega _{1},\omega _{2}) &=&K_{j}^{\alpha }(\omega _{1},\omega
_{2})\chi _{B_{j}}(\omega _{1})\chi _{B_{j}^{c}}(\omega _{2}), \\
K_{j}^{3}(\omega _{1},\omega _{2}) &=&K_{j}^{\alpha }(\omega _{1},\omega
_{2})\chi _{B_{j}^{c}}(\omega _{1})\chi _{B_{j}}(\omega _{2}), \\
K_{j}^{4}(\omega _{1},\omega _{2}) &=&K_{j}^{\alpha }(\omega _{1},\omega
_{2})\chi _{B_{j}^{c}}(\omega _{1})\chi _{B_{j}^{c}}(\omega _{2}).
\end{eqnarray*}
Here $\chi _{A}$ stands for the characteristic function of set $A$ and $
\gamma >0$ is to be fixed. Let $T_{j}^{l}$ be the bilinear operator with
kernel $K_{j}^{l}$, $l=1,2,3,4$. Then, (\ref{Tj-}) would be the consequence
of the estimates
\begin{equation*}
\left\Vert T_{j}^{l}\right\Vert _{L^{p_{1}}\times L^{p_{2}}\rightarrow
L^{p}}\leq C2^{-\varepsilon j},\quad l=1,2,3,4.
\end{equation*}

We first consider $T_{j}^{4}$. Set $R_{t}^{l}(\omega )$ to be the kernel of
the Riesz means $\int_{0}^{t}(1-\frac{\lambda }{t})^{l}P_{\lambda }\,d\mu
(\lambda )$. Then,
\begin{equation*}
t\rightarrow R_{t}^{0}(\omega )
\end{equation*}
is a function of bounded variation and the kernel of $T_{j}^{\alpha }$ can
be written as
\begin{equation*}
K_{j}^{\alpha }(\omega _{1},\omega _{2})=\int \varphi _{j}^{\alpha }(s,t)
\frac{\partial }{\partial s}R_{s}^{0}(\omega _{1})\frac{\partial }{\partial
t }R_{t}^{0}(\omega _{2})\,dsdt.
\end{equation*}
Intergrating by parts and using the identity
\begin{equation*}
\frac{\partial }{\partial t}\big(t^{m}R_{t}^{m}(\omega )\big)
=mt^{m-1}R_{t}^{m-1}(\omega ),
\end{equation*}
where $m$ is a positive integer, we get the relation
\begin{equation*}
K_{j}^{\alpha }(\omega _{1},\omega _{2})=c_{m}\int \left(\left( \partial
_{s}\partial _{t}\right) ^{2m+2}\varphi _{j}^{\alpha
}(s,t)\right)s^{2m+1}R_{s}^{2m+1}(\omega _{1})t^{2m+1}R_{t}^{2m+1}(\omega
_{2})\,dsdt.
\end{equation*}
Note that (see Theorem 2.5.3 in \cite{Thang})
\begin{equation}
\left\vert R_{t}^{2m+1}(\omega )\right\vert \leq Ct^{n+1}(1+t^{\frac{1}{2}
}\left\vert \omega \right\vert )^{-2m},  \label{Rt}
\end{equation}
This estimate, together with the bound
\begin{equation*}
\left\vert \left( \partial _{s}\partial _{t}\right) ^{2m+2}\varphi
_{j}^{\alpha }\right\vert \leq C2^{j(4m+4)}
\end{equation*}
imply that
\begin{equation*}
\left\vert K_{j}^{\alpha }(\omega _{1},\omega _{2})\right\vert \leq
C2^{j(4m+4)}(1+\left\vert \omega _{1}\right\vert )^{-2m}(1+\left\vert \omega
_{2}\right\vert )^{-2m}. 
\end{equation*}
So, we can use H\"{o}lder's inequality and Young's inequality to get that
\begin{eqnarray*}
\left\Vert T_{j}^{4}(f,g)\right\Vert _{p} &\leq &C\left\Vert f\right\Vert
_{p_{1}}\left\Vert g\right\Vert _{p_{2}}2^{j(4m+4)}\int_{\left\vert \omega
_{1}\right\vert \geq 2^{j(1+\gamma )}}(1+\left\vert \omega _{1}\right\vert
)^{-2m}d\omega _{1} \\
&&\times \int_{\left\vert \omega _{2}\right\vert \geq 2^{j(1+\gamma
)}}(1+\left\vert \omega _{2}\right\vert )^{-2m}d\omega _{2} \\
&\leq &C2^{j(4m+4)}2^{j(1+\gamma )(-4m+2Q)}\left\Vert f\right\Vert
_{p_{1}}\left\Vert g\right\Vert _{p_{2}}.
\end{eqnarray*}
Choosing $m$ large enough such that
\begin{equation*}
4m\gamma \geq 2Q(1+\gamma )+4,
\end{equation*}
we have
\begin{equation}
\left\Vert T_{j}^{4}\right\Vert _{L^{p_{1}}\times L^{p_{2}}\rightarrow
L^{p}}\leq C2^{-\varepsilon j}  \label{Tj4}
\end{equation}
for some $\varepsilon >0$.

Next, we consider the estimate of $T_{j}^{3}$. Note that $K_{j}^{\alpha }$
also can be written as
\begin{equation*}
K_{j}^{\alpha }(\omega _{1},\omega _{2})=c_{m}\int_{0}^{1}\int_{0}^{1}\left(
\partial _{\lambda _{1}}^{2m+2}\varphi _{j}^{\alpha }(\lambda _{1},\lambda
_{2})\right) \lambda _{1}^{2m+1}R_{\lambda _{1}}^{2m+1}(\omega
_{1})G_{\lambda _{2}}(\omega _{2})\,d\lambda _{1}d\mu (\lambda _{2})
\end{equation*}
where
\begin{equation*}
G_{\lambda _{2}}(\omega _{2})=G_{\lambda _{2}}(z,t)=\sum_{k=0}^{\infty
}(2k+n)^{-n-1}(\widetilde{e}_{k}^{\lambda _{2}}+\widetilde{e}_{k}^{-\lambda
_{2}})(z,t)
\end{equation*}
is the kernel of the projection operator $P_{\lambda _{2}}$. Then, it
follows that
\begin{eqnarray*}
&&\left\vert K_{j}^{3}(\omega _{1},\omega _{2})\right\vert  \\
&=&\left\vert K_{j}^{\alpha }(\omega _{1},\omega _{2})\chi
_{B_{j}^{c}}(\omega _{1})\chi _{B_{j}}(\omega _{2})\right\vert  \\
&\leq &C\int_{0}^{1}\left\vert \lambda _{1}^{2m+1}R_{\lambda
_{1}}^{2m+1}(\omega _{1})\chi _{B_{j}^{c}}(\omega _{1})\right\vert
\left\vert \int_{0}^{1}\left( \partial _{\lambda _{2}}^{2m+2}\varphi
_{j}^{\alpha }(\lambda _{1},\lambda _{2})\right) G_{\lambda _{2}}(\omega
_{2})\chi _{B_{j}}(\omega _{2})\,d\mu (\lambda _{2})\right\vert d\lambda _{1}
\\
&\leq &C\sup_{\lambda _{1}\in \lbrack 0,1]}\left\vert \lambda
_{1}^{2m+1}R_{\lambda _{1}}^{2m+1}(\omega _{1})\chi _{B_{j}^{c}}(\omega
_{1})\right\vert \int_{0}^{1}\left\vert \int_{0}^{1}\left( \partial
_{\lambda _{2}}^{2m+2}\varphi _{j}^{\alpha }(\lambda _{1},\lambda
_{2})\right) G_{\lambda _{2}}(\omega _{2})\chi _{B_{j}}(\omega _{2})\,d\mu
(\lambda _{2})\right\vert d\lambda _{1}.
\end{eqnarray*}
We get
\begin{eqnarray*}
\left\Vert T_{j}^{3}(f,g)\right\Vert _{p} &\leq &C\left\Vert f\right\Vert
_{p_{1}}\left\Vert g\right\Vert _{p_{2}}\int_{\left\vert \omega
_{1}\right\vert \geq 2^{j(1+\gamma )}}\sup_{\lambda _{1}\in \lbrack 0,1]}
\Big\vert\lambda _{1}^{2m+1}R_{\lambda _{1}}^{2m+1}(\omega _{1})\Big \vert
\,d\omega _{1} \\
&&\times \int_{\left\vert \omega _{2}\right\vert \leq 2^{j(1+\gamma
)}}\int_{0}^{1}\left\vert \int_{0}^{1}\left( \partial _{\lambda
_{2}}^{2m+2}\varphi _{j}^{\alpha }(\lambda _{1},\lambda _{2})\right)
G_{\lambda _{2}}(\omega _{2})\chi _{B_{j}}(\omega _{2})\,d\mu (\lambda
_{2})\right\vert d\lambda _{1}d\omega _{2}.
\end{eqnarray*}
Applying the Plancherel theorem in the variable $t$, the orthogonality of $
\varphi _{k}^{\lambda }$ and (\ref{phi2}), we have that
\begin{eqnarray*}
&&\int_{0}^{1}\int_{\left\vert \omega _{2}\right\vert \leq 2^{j(1+\gamma
)}}\left\vert \int_{0}^{1}\left( \partial _{\lambda _{2}}^{2m+2}\varphi
_{j}^{\alpha }(\lambda _{1},\lambda _{2})\right) G_{\lambda _{2}}(\omega
_{2})\,d\mu (\lambda _{2})\right\vert d\omega _{2}d\lambda _{1} \\
&\leq &C2^{j\frac{(1+\gamma )Q}{2}}\int_{0}^{1}\left( \int_{\mathbb{C}
^{n}}\int_{\mathbb{R}}\left\vert \int_{0}^{1}\left( \partial _{\lambda
_{2}}^{2m+2}\varphi _{j}^{\alpha }(\lambda _{1},\lambda _{2})\right)
G_{\lambda _{2}}(\omega _{2})\,d\mu (\lambda _{2})\right\vert
^{2}dtdz\right) ^{\frac{1}{2}}d\lambda _{1} \\
&\leq &C2^{j\frac{(1+\gamma )Q}{2}}\int_{0}^{1}\left( \int_{\mathbb{C}
^{n}}\int_{\mathbb{R}}\left\vert \sum_{k=0}^{\infty }\int_{0}^{\frac{1}{2k+n}
}e^{-i\lambda _{2}t}\left( \partial _{\lambda _{2}}^{2m+2}\varphi
_{j}^{\alpha }(\lambda _{1},(2k+n)\left\vert \lambda _{2}\right\vert
)\right) \varphi _{k}^{\lambda _{2}}(z)\,d\mu (\lambda _{2})\right\vert
^{2}dtdz\right) ^{\frac{1}{2}}d\lambda _{1} \\
&\leq &C2^{j\frac{(1+\gamma )Q}{2}}\int_{0}^{1}\left( \int_{\mathbb{R}}\int_{
\mathbb{C}^{n}}\left\vert \sum_{k=0}^{\infty }\chi _{\lbrack 0,\frac{1}{2k+n}
]}(\lambda _{2})\left( \partial _{\lambda _{2}}^{2m+2}\varphi _{j}^{\alpha
}(\lambda _{1},(2k+n)\left\vert \lambda _{2}\right\vert )\right) \varphi
_{k}^{\lambda _{2}}(z)\,\right\vert ^{2}dz\left\vert \lambda _{2}\right\vert
^{2n}d\lambda _{2}\right) ^{\frac{1}{2}}d\lambda _{1} \\
&=&C2^{j\frac{(1+\gamma )Q}{2}}\int_{0}^{1}\left( \sum_{k=0}^{\infty
}\int_{0}^{\frac{1}{2k+n}}\left\vert \partial _{\lambda _{2}}^{2m+2}\varphi
_{j}^{\alpha }(\lambda _{1},(2k+n)\left\vert \lambda _{2}\right\vert
)\right\vert ^{2}\left\Vert \varphi _{k}^{\lambda }\,\right\Vert
_{2}^{2}\left\vert \lambda _{2}\right\vert ^{2n}\,d\lambda _{2}\right) ^{
\frac{1}{2}}d\lambda _{1} \\
&\leq &C2^{j\frac{(1+\gamma )Q}{2}}\int_{0}^{1}\left( \sum_{k=0}^{\infty
}(2k+n)^{-n-1}k^{n-1}\int_{0}^{1}\left\vert \partial _{\lambda
_{2}}^{2m+2}\varphi _{j}^{\alpha }(\lambda _{1},\left\vert \lambda
_{2}\right\vert )\right\vert ^{2}\left\vert \lambda _{2}\right\vert
^{n}\,d\lambda _{2}\right) ^{\frac{1}{2}}d\lambda _{1} \\
&\leq &C2^{j\frac{(1+\gamma )Q}{2}}\int_{0}^{1}\left( \int_{0}^{1}\left\vert
\partial _{\lambda _{2}}^{2m+2}\varphi _{j}^{\alpha }(\lambda
_{1},\left\vert \lambda _{2}\right\vert )\right\vert ^{2}\left\vert \lambda
_{2}\right\vert ^{n}\,d\lambda _{2}\right) ^{\frac{1}{2}}d\lambda _{1} \\
&\leq &C2^{j\frac{(1+\gamma )Q}{2}}2^{j(2m+2)}.
\end{eqnarray*}
On the other hand, from (\ref{Rt}), we have
\begin{equation*}
\sup_{\lambda _{1}\in \lbrack 0,1]}\left\vert \lambda _{1}^{2m+1}R_{\lambda
_{1}}^{2m+1}(\omega _{1})\right\vert \leq C(1+\left\vert \omega
_{1}\right\vert )^{-2m},
\end{equation*}
which implies that
\begin{eqnarray*}
\int_{\left\vert \omega _{1}\right\vert \geq 2^{j(1+\gamma )}}\sup_{\lambda
_{1}\in \lbrack 0,1]}\left\vert \lambda _{1}^{2m+1}R_{\lambda
_{1}}^{2m+1}(\omega _{1})\right\vert d\omega _{1} &\leq &C\int_{\left\vert
\omega _{1}\right\vert \geq 2^{j(1+\gamma )}}(1+\left\vert \omega
_{1}\right\vert )^{-2m}d\omega _{1} \\
&\leq &C2^{j(1+\gamma )(-2m+Q)}.
\end{eqnarray*}
Thus,
\begin{equation*}
\left\Vert T_{j}^{3}(f,g)\right\Vert _{p}\leq C\left\Vert f\right\Vert
_{p_{1}}\left\Vert g\right\Vert _{p_{2}}2^{j\frac{(1+\gamma )Q}{2}
}2^{j(2m+2)}2^{j(1+\gamma )(-2m+Q)}\text{,}
\end{equation*}
and we can choose $m$ large enough such that
\begin{equation*}
2m\gamma \geq \frac{3}{2}Q(1+\gamma )+2,
\end{equation*}
which yields that
\begin{equation}
\left\Vert T_{j}^{3}\right\Vert _{L^{p_{1}}\times L^{p_{2}}\rightarrow
L^{p}}\leq C2^{-\varepsilon j}  \label{Tj3}
\end{equation}
for some $\varepsilon >0$. Obviously, (\ref{Tj3}) also holds for $T_{j}^{2}$.

Now, it remains to estimate $T_{j}^{1}$. For $\xi \in \mathbb{H}^{n}$, we
set $B_{j}(\xi ,R)=\big\{\omega :\left\vert \xi ^{-1}\omega \right\vert \leq
R2^{j(1+\gamma )}\big\}$ with $R>0$, and slipt the functions $f$ and $g$
into three parts respectively: $f=f_{1}+f_{2}+f_{3}$, $g=g_{1}+g_{2}+g_{3}$,
where
\begin{eqnarray*}
f_{1} &=&f\chi _{B_{j}(\xi ,\frac{3}{4})},\qquad \qquad g_{1}=g\chi
_{B_{j}(\xi ,\frac{3}{4})}, \\
f_{2} &=&f\chi _{B_{j}(\xi ,\frac{5}{4})\backslash B_{j}(\xi ,\frac{3}{4}
)},\quad g_{2}=g\chi _{B_{j}(\xi ,\frac{5}{4})\backslash B_{j}(\xi ,\frac{3}{
4})}, \\
f_{3} &=&f\chi _{\mathbb{H}^{n}\backslash B_{j}(\xi ,\frac{5}{4})},\qquad \
\ g_{3}=g\chi _{\mathbb{H}^{n}\backslash B_{j}(\xi ,\frac{5}{4})}.
\end{eqnarray*}
Assume that $\left\vert \xi ^{-1}\omega \right\vert \leq \frac{1}{4}
2^{j(1+\gamma )}$. Since that $f_{3}$ is supported on $\mathbb{H}
^{n}\backslash B_{j}(\xi ,\frac{5}{4})$, then $f_{3}\neq 0$ leads to
\begin{equation*}
\left\vert \xi ^{-1}\omega _{1}\right\vert \geq \frac{5}{4}2^{j(1+\gamma )}.
\end{equation*}
It follows that
\begin{equation*}
\left\vert \omega _{1}^{-1}\omega \right\vert \geq 2^{j(1+\gamma )}.
\end{equation*}
Note that the kernel $K_{j}^{1}$ is supported on $B_{j}\times B_{j}$. Hence,
$T_{j}^{1}(f_{3},g)=0$. In the same way, $T_{j}^{1}(f,g_{3})=0$. Since that $
f_{2}$ and $g_{2}$ are supported on $B_{j}(\xi ,\frac{5}{4})\backslash
B_{j}(\xi ,\frac{3}{4})$, then $f_{2},g_{2}\neq 0$ yields that
\begin{equation*}
\left\vert \omega _{1}^{-1}\omega \right\vert \geq \frac{1}{2}2^{j(1+\gamma
)}\text{ and }\left\vert \omega _{2}^{-1}\omega \right\vert \geq \frac{1}{2}
2^{j(1+\gamma )}.
\end{equation*}
Repeating the proof of (\ref{Tj4}), we get
\begin{eqnarray}
\left\Vert T_{j}^{1}(f_{2},g_{2})\right\Vert _{L^{p}(B_{j}(\xi ,\frac{1}{4}
))} &\leq &C2^{-\varepsilon j}\left\Vert f_{2}\right\Vert _{p_{1}}\left\Vert
g_{2}\right\Vert _{p_{2}}  \label{f7} \\
&\leq &C2^{-\varepsilon j}\left\Vert f\right\Vert _{L^{p_{1}}(B_{j}(\xi ,
\frac{5 }{4}))}\left\Vert g\right\Vert _{L^{p_{2}}(B_{j}(\xi ,\frac{5}{4}))}.
\notag
\end{eqnarray}
Taking the $L^{p}$ norm with respect to $\xi $ on the both side of (\ref{f7}) and using H\"{o}lder's inequality, we have that
\begin{eqnarray}
&&\left( \int_{\mathbb{H}^{n}}\int_{B_{j}(\xi ,\frac{1}{4})}\left\vert
T_{j}^{1}(f_{2},g_{2})(\omega )\right\vert ^{p}d\omega d\xi \right) ^{\frac{
1 }{p}}  \notag \\
&\leq &C2^{-\varepsilon j}\left( \int_{\mathbb{H}^{n}}\int_{B_{j}(\xi ,\frac{
5}{4 })}\left\vert f(\omega )\right\vert ^{p_{1}}d\omega d\xi \right) ^{
\frac{1}{ p_{1}}}\left( \int_{\mathbb{H}^{n}}\int_{B_{j}(\xi ,\frac{5}{4}
)}\left\vert g(\omega )\right\vert ^{p_{2}}d\omega d\xi \right) ^{\frac{1}{
p_{2}}}.  \notag
\end{eqnarray}
Changing variable and exchanging the order of integration, the left side
equals to
\begin{eqnarray*}
\left( \int_{\mathbb{H}^{n}}\int_{\left\vert \omega \right\vert \leq \frac{1
}{4}2^{j(1+\gamma )}}\left\vert T_{j}^{1}(f_{2},g_{2})(\xi \omega
)\right\vert ^{p}d\omega d\xi \right) ^{\frac{1}{p}} &=&\left(
\int_{\left\vert \omega \right\vert \leq \frac{1}{4}2^{j(1+\gamma )}}\int_{
\mathbb{H}^{n}}\left\vert T_{j}^{1}(f_{2},g_{2})(\xi \omega )\right\vert
^{p}d\xi d\omega \right) ^{\frac{1}{p}} \\
&=&\left( \frac{1}{4}2^{j(1+\gamma )}\right) ^{\frac{Q}{p}}\left\Vert
T_{j}^{1}(f_{2},g_{2})\right\Vert _{p},
\end{eqnarray*}
and the right side equals to
\begin{equation*}
C2^{-\varepsilon j}\left( \frac{5}{4}2^{j(1+\gamma )}\right) ^{\frac{Q}{
p_{1} }}\left\Vert f\right\Vert _{p_{1}}\left( \frac{5}{4}2^{j(1+\gamma
)}\right) ^{\frac{Q}{p_{2}}}\left\Vert g\right\Vert
_{p_{2}}=C2^{-\varepsilon j}\left( \frac{5}{4}2^{j(1+\gamma )}\right) ^{
\frac{Q}{p}}\left\Vert f\right\Vert _{p_{1}}\left\Vert g\right\Vert _{p_{2}}.
\end{equation*}
This yields that
\begin{equation}
\left\Vert T_{j}^{1}(f_{2},g_{2})\right\Vert _{p}\leq C2^{-\varepsilon
j}\left\Vert f\right\Vert _{p_{1}}\left\Vert g\right\Vert _{p_{2}}.
\label{f3}
\end{equation}
Since $f_{1}$ is supported on $B_{j}(\xi ,\frac{3}{4})$, then $
f_{1},g_{2}\neq 0$ implies that
\begin{equation*}
\left\vert \omega _{1}^{-1}\omega \right\vert \leq 2^{j(1+\gamma )}\text{
and }\left\vert \omega _{2}^{-1}\omega \right\vert \geq \frac{1}{2}
2^{j(1+\gamma )}.
\end{equation*}
We can repeat the proof of (\ref{Tj3}) to get that
\begin{eqnarray*}
\left\Vert T_{j}^{1}(f_{1},g_{2})\right\Vert _{L^{p}(B_{j}(\xi ,\frac{1}{4}
))} &\leq &C2^{-\varepsilon j}\left\Vert f_{1}\right\Vert _{p_{1}}\left\Vert
g_{2}\right\Vert _{p_{2}} \\
&\leq &C2^{-\varepsilon j}\left\Vert f\right\Vert _{L^{p_{1}}(B_{j}(\xi ,
\frac{3 }{4}))}\left\Vert g\right\Vert _{L^{p_{2}}(B_{j}(\xi ,\frac{5}{4}))}.
\end{eqnarray*}
Taking the $L^{p}$ norm with respect to $\xi $ as above, we have
\begin{equation}
\left\Vert T_{j}^{1}(f_{1},g_{2})\right\Vert _{p}\leq C2^{-\varepsilon
j}\left\Vert f\right\Vert _{p_{1}}\left\Vert g\right\Vert _{p_{2}}.
\label{f2}
\end{equation}
Obviously, (\ref{f2}) also holds for $T_{j}^{1}(f_{2},g_{1})$. Finally, we
consider $T_{j}^{1}(f_{1},g_{1})$. Because $f_{1},g_{1}\neq 0$ implies that
\begin{equation*}
\left\vert \omega _{1}^{-1}\omega \right\vert \leq 2^{j(1+\gamma )}\text{
and }\left\vert \omega _{2}^{-1}\omega \right\vert \leq 2^{j(1+\gamma )},
\end{equation*}
we have
\begin{equation}
T_{j}^{1}(f_{1},g_{1})(\omega )=T_{j}^{\alpha }(f_{1},g_{1})(\omega ).
\label{f1}
\end{equation}
Note that $T_{j}^{\alpha }$ can be written as
\begin{eqnarray*}
T_{j}^{\alpha }(f,g) &=&\int_{0}^{\infty }\int_{0}^{\infty }\varphi
_{j}^{\alpha }\left( \lambda _{1},\lambda _{2}\right) P_{\lambda
_{1}}fP_{\lambda _{2}}g\,d\mu (\lambda _{1})d\mu (\lambda _{2}) \\
&=&C\int_{[-1,1]^{2}}\varphi _{j}^{\alpha }\left( \left\vert \lambda
_{1}\right\vert ,\left\vert \lambda _{2}\right\vert \right) P_{\left\vert
\lambda _{1}\right\vert }fP_{\left\vert \lambda _{2}\right\vert }g\,d\mu
(\lambda _{1})d\mu (\lambda _{2})\text{.}
\end{eqnarray*}
Because for any fixed $s\in \lbrack -1,1]$, the function
\begin{equation*}
t\rightarrow \varphi _{j}^{\alpha }\left( \left\vert s\right\vert
,\left\vert t\right\vert \right)
\end{equation*}
is supported in $[-1,1]$ and vanishes at endpoints $\pm 1$, so we can expand
this function in Fourier series by considering a periodic extension on $
\mathbb{R}$ of period $2$. Then, we have
\begin{equation*}
\varphi _{j}^{\alpha }(\left\vert s\right\vert ,\left\vert t\right\vert
)=\sum_{k\in \mathbb{Z}}\gamma _{j,k}^{\alpha }(s)e^{i\pi kt}
\end{equation*}
with Fourier coefficients
\begin{equation*}
\gamma _{j,k}^{\alpha }(s)=\frac{1}{2} \int_{-1}^{1}\varphi _{j}^{\alpha
}(\left\vert s\right\vert ,\left\vert t\right\vert )e^{-i\pi kt}\,dt\text{.}
\end{equation*}
It is easy to see that for any $0< \delta < \alpha$,
\begin{equation*}
\sup_{s\in \lbrack -1,1]}\left\vert \gamma _{j,k}^{\alpha }(s)\right\vert
\left( 1+\left\vert k\right\vert \right) ^{1+\delta}\leq C2^{-j(\alpha
-\delta)}.
\end{equation*}
$T_{j}^{\alpha }$ can be expressed by
\begin{eqnarray*}
T_{j}^{\alpha }(f,g) &=&C\int_{[-1,1]^{2}}\varphi _{j}^{\alpha }\left(
\left\vert \lambda _{1}\right\vert ,\left\vert \lambda _{2}\right\vert
\right) P_{\left\vert \lambda _{1}\right\vert }fP_{\left\vert \lambda
_{2}\right\vert }g\,d\mu (\lambda _{1})d\mu (\lambda _{2}) \\
&=&C\sum_{k\in \mathbb{Z}}\int_{[-1,1]^{2}}\gamma _{j,k}^{\alpha }(\lambda
_{1})e^{i\pi k\lambda _{2}}P_{\left\vert \lambda _{1}\right\vert
}fP_{\left\vert \lambda _{2}\right\vert }g\,d\mu (\lambda _{1})d\mu (\lambda
_{2}) \\
&=&C\sum_{k\in \mathbb{Z}}\int_{-1}^{1}\gamma _{j,k}^{\alpha }(\lambda
_{1})P_{\left\vert \lambda _{1}\right\vert }f\,d\mu (\lambda
_{1})\int_{-1}^{1}e^{i\pi k\lambda _{2}}P_{\left\vert \lambda
_{2}\right\vert }g\,d\mu (\lambda _{2}).
\end{eqnarray*}
Applying Cauchy-Schwartz's inequality and Lemma \ref{restriction}, we have
\begin{eqnarray}
\left\Vert T_{j}^{\alpha }(f,g)\right\Vert _{1} &\leq &C\sum_{k\in \mathbb{Z}
}\left\Vert \int_{-1}^{1}\gamma _{j,k}^{\alpha }(\lambda _{1})P_{\left\vert
\lambda _{1}\right\vert }f\,d\mu (\lambda _{1})\right\Vert _{2}\left\Vert
\int_{-1}^{1}e^{i\pi k\lambda _{2}}P_{\left\vert \lambda _{2}\right\vert
}g\,d\mu (\lambda _{2})\right\Vert _{2}  \label{ff} \\
&\leq &C\sum_{k\in \mathbb{Z}}(1+\left\vert k\right\vert )^{-1-\delta}\left(
\sup_{s\in \lbrack -1,1]}\left\vert \gamma _{j,k}^{\alpha }(s)\right\vert
\left( 1+\left\vert k\right\vert \right) ^{1+\delta}\right) \left\Vert
f\right\Vert _{p_{1}}\left\Vert g\right\Vert _{p_{2}}  \notag \\
&\leq &C2^{-j(\alpha -\delta)}\left\Vert f\right\Vert _{p_{1}}\left\Vert
g\right\Vert _{p_{2}}.  \notag
\end{eqnarray}
Then, using H\"{o}lder's inequality and (\ref{f1}), it follows that
\begin{eqnarray*}
\left\Vert T_{j}^{1}(f_{1},g_{1})\right\Vert _{L^{p}(B_{j}(\xi ,\frac{1}{4}
))} &\leq &2^{j(1+\gamma )Q\left( \frac{1}{p}-1\right) }\left\Vert
T_{j}^{1}(f_{1},g_{1})\right\Vert _{L^{1}(B_{j}(\xi ,\frac{1}{4}))} \\
&=&2^{j(1+\gamma )Q\left( \frac{1}{p}-1\right) }\left\Vert T_{j}^{\alpha
}(f_{1},g_{1})\right\Vert _{L^{1}(B_{j}(\xi ,\frac{1}{4}))} \\
&\leq &C2^{-j(\alpha -\delta)}2^{j(1+\gamma )Q\left( \frac{1}{ p}-1\right)
}\left\Vert f_{1}\right\Vert _{p_{1}}\left\Vert g_{1}\right\Vert _{p_{2}} \\
&\leq &C2^{-j(\alpha -\delta)}2^{j(1+\gamma )Q\left( \frac{1}{ p}-1\right)
}\left\Vert f\right\Vert _{L^{p_{1}}(B_{j}(\xi ,\frac{3}{4} ))}\left\Vert
g\right\Vert _{L^{p_{2}}(B_{j}(\xi ,\frac{3}{4}))}.
\end{eqnarray*}
Taking the $L^{p}$ norm with respect to $\xi $ yields that
\begin{equation}
\left\Vert T_{j}^{1}(f_{1},g_{1})\right\Vert _{p}\leq C2^{-j(\alpha
-\delta)} 2^{j(1+\gamma )Q\left( \frac{1}{p}-1\right) }\left\Vert
f\right\Vert _{p_{1}}\left\Vert g\right\Vert _{p_{2}}.  \label{f5}
\end{equation}
Combining (\ref{f3}), (\ref{f2}) and (\ref{f5}), we conclude that
\begin{equation*}
\left\Vert T_{j}^{1}(f,g)\right\Vert _{p}\leq C2^{-j(\alpha -\delta)}
2^{j(1+\gamma )Q\left( \frac{1}{p}-1\right) }\left\Vert f\right\Vert
_{p_{1}}\left\Vert g\right\Vert _{p_{2}}.
\end{equation*}
Thus, whenever $\alpha >Q\left( \frac{1}{p}-1\right) $, we can choose $
\gamma, \delta >0$ such that
\begin{equation*}
\alpha >Q(1+\gamma )\left( \frac{1}{p}-1\right) +\delta,
\end{equation*}
which implies that there exists an $\varepsilon >0$ such that
\begin{equation*}
\left\Vert T_{j}^{1}\right\Vert _{L^{1}\times L^{p_{2}}\rightarrow
L^{p}}\leq 2^{-\varepsilon j}.
\end{equation*}
The proof of Theorem \ref{mainTh} is completed.
\end{proof}

\section{Boundedness of $S^{\protect\alpha}$ for particular points}

In this section, we investigate the boundedness of $S^{\alpha }$ for some
specific triples of points $(p_{1},p_{2},p)$.

\subsection{The point $(1,\infty ,1)$}

\begin{theorem}
\label{1infty} If $\alpha >\frac{Q}{2}$, then $S^{\alpha }$ is bounded from $
L^{1}(\mathbb{H}^{n})\times L^{\infty }(\mathbb{H}^{n})$ to $L^{1}(\mathbb{H}
^{n})$.
\end{theorem}

\begin{proof}
\quad We keep the notations in Section 4. Note that (\ref{Tj4}), (\ref{Tj3}), (\ref{f3}) and (\ref{f2}) hold for any $\alpha >0$, the proof of Theorem
\ref{mainTh} is valid apart from the estimate of $T_{j}^{1}(f_{1},g_{1})$.
According to (\ref{ff}), for any $0< \delta < \alpha$,
\begin{equation*}
\left\Vert T_{j}^{\alpha }(f,g)\right\Vert _{1}\leq C2^{-j(\alpha -\delta
)}\left\Vert f\right\Vert _{1}\left\Vert g\right\Vert _{2},
\end{equation*}
we have
\begin{eqnarray*}
\left\Vert T_{j}^{1}(f_{1},g_{1})\right\Vert _{L^{1}(B_{j}(\xi ,\frac{1}{4}
))} &=&\left\Vert T_{j}^{\alpha }(f_{1},g_{1})\right\Vert _{L^{1}(B_{j}(\xi ,
\frac{1 }{4}))} \\
&\leq &C2^{-j(\alpha -\delta)}\left\Vert f\right\Vert _{L^{1}(B_{j}(\xi ,
\frac{3}{4} ))}\left\Vert g\right\Vert _{L^{2}(B_{j}(\xi ,\frac{3}{4}))} \\
&\leq &C2^{-j(\alpha -\delta)}2^{j(1+\gamma )\frac{Q}{2}}\left\Vert
f\right\Vert _{L^{1}(B_{j}(\xi ,\frac{3}{4}))}\left\Vert g\right\Vert
_{L^{\infty }(B_{j}(\xi ,\frac{1}{4}))}.
\end{eqnarray*}
It follows that
\begin{equation*}
\left\Vert T_{j}^{1}(f_{1},g_{1})\right\Vert _{1}\leq C2^{-j(\alpha -\delta
)}2^{j(1+\gamma )\frac{Q}{2}}\left\Vert f\right\Vert _{1}\left\Vert
g\right\Vert _{\infty }.
\end{equation*}%
Thus, when $\alpha >\frac{Q}{2}$, we can choose $\gamma, \delta >0$ such
that $\alpha >\frac{(1+\gamma )Q}{2}+\delta $, which yields that there
exists $\varepsilon >0$ such that
\begin{equation*}
\left\Vert T_{j}^{1}(f_{1},g_{1})\right\Vert _{1}\leq C2^{-\varepsilon
j}\left\Vert f\right\Vert _{1}\left\Vert g\right\Vert _{\infty }.
\end{equation*}
The proof is completed.
\end{proof}

\subsection{The point $(\infty ,\infty ,\infty )$}

\begin{theorem}
\label{infty}If $\alpha >Q- \frac{1}{2}$, then $S^{\alpha }$ is bounded from
$L^{\infty }( \mathbb{H}^{n})\times L^{\infty }(\mathbb{H}^{n})$ into $
L^{\infty }(\mathbb{\ H}^{n})$.
\end{theorem}

\begin{proof}
\quad We still keep the notion in Section 4. To prove this theorem, it
suffices to estimate $T_{j}^{1}(f_{1},g_{1})$. Since that
\begin{equation*}
P_{\lambda }f(z,t)=\sum_{k=0}^{\infty }(2k+n)^{-n-1}f\ast (\widetilde{e}
_{k}^{\lambda }+\widetilde{e}_{k}^{-\lambda })(z,t)\text{, \ \ }\lambda >0,
\end{equation*}
and
\begin{equation*}
f\ast e_{k}^{\lambda }=e^{-i\lambda t}f^{\lambda }\ast _{\lambda }\varphi
_{k}^{\lambda },
\end{equation*}
we can write $T_{j}^{\alpha }(f,g)$ as
\begin{eqnarray*}
T_{j}^{\alpha }(f,g)(z,t) &=&\int_{0}^{\infty }\int_{0}^{\infty }\varphi
_{j}^{\alpha }(\lambda _{1},\lambda _{2})\sum_{k=0}^{\infty
}(2k+n)^{-n-1}f\ast (\widetilde{e}_{k}^{\lambda _{1}}+\widetilde{e}
_{k}^{-\lambda _{1}})(z,t) \\
&&\times \sum_{l=0}^{\infty }(2l+n)^{-n-1}g\ast (\widetilde{e}_{l}^{\lambda
_{2}}+\widetilde{e}_{l}^{-\lambda _{2}})(z,t)\,d\mu (\lambda _{1})d\mu
(\lambda _{2}) \\
&=&C\int_{\mathbb{R}}\int_{\mathbb{R} }\sum_{k=0}^{\infty
}\sum_{l=0}^{\infty }e^{-i(\lambda _{1}+\lambda _{2})t}\varphi _{j}^{\alpha
}\left( (2k+n)\left\vert \lambda _{1}\right\vert ,(2l+n)\left\vert \lambda
_{2}\right\vert \right) \\
&&\times \left( f^{\lambda _{1}}\ast _{\lambda _{1}}\varphi _{k}^{\lambda
_{1}}\right) (z)\left( g^{\lambda _{2}}\ast _{\lambda _{2}}\varphi
_{l}^{\lambda _{2}}\right) (z)\left\vert \lambda _{1}\right\vert
^{n}\left\vert \lambda _{2}\right\vert ^{n}\,d\lambda _{1}d\lambda _{2}.
\end{eqnarray*}
Using again (\ref{phi2}), for $0<\delta<1$, we get
\begin{eqnarray*}
&&\left\Vert T_{j}^{\alpha }(f,g)\right\Vert _{\infty } \\
&\leq &C\int_{\mathbb{R}}\int_{\mathbb{R}}\sum_{k=0}^{\infty
}\sum_{l=0}^{\infty } \left\vert \varphi _{j}^{\alpha
}\left((2k+n)\left\vert \lambda _{1}\right\vert , (2l+n)\left\vert
\lambda_{2}\right\vert \right) \right\vert \\
&&\times \left\Vert f^{\lambda _{1}}\ast _{\lambda _{1}}\varphi
_{k}^{\lambda _{1}}\right\Vert _{\infty }\left\Vert g^{\lambda _{2}}\ast
_{\lambda _{2}}\varphi _{l}^{\lambda _{2}}\right\Vert _{\infty }\left\vert
\lambda _{1}\right\vert ^{n}\left\vert \lambda _{2}\right\vert
^{n}\,d\lambda _{1}d\lambda _{2} \\
&\leq &C\int_{\mathbb{R}}\int_{\mathbb{R}}\sum_{k=0}^{\infty
}\sum_{l=0}^{\infty } \left\vert \varphi _{j}^{\alpha
}\left((2k+n)\left\vert \lambda _{1}\right\vert , (2l+n)\left\vert
\lambda_{2}\right\vert \right) \right\vert \\
&&\times \left\Vert f^{\lambda _{1}}\right\Vert _{2}\left\Vert \varphi
_{k}^{\lambda _{1}}\right\Vert _{2}\left\Vert g^{\lambda _{2}}\right\Vert
_{2}\left\Vert \varphi _{l}^{\lambda _{2}}\right\Vert _{2}\left\vert \lambda
_{1}\right\vert ^{n}\left\vert \lambda _{2}\right\vert ^{n}\,d\lambda
_{1}d\lambda _{2} \\
&\leq &C\int_{\mathbb{R}}\int_{\mathbb{R}}\sum_{k=0}^{\infty
}\sum_{l=0}^{\infty } \left\vert \varphi _{j}^{\alpha
}\left((2k+n)\left\vert \lambda _{1}\right\vert , (2l+n)\left\vert
\lambda_{2}\right\vert \right) \right\vert \left\vert \lambda _{1}
\right\vert ^{\frac{n}{2}}k^{\frac{n-1}{2}}\left\vert \lambda
_{2}\right\vert ^{\frac{n}{2}}l^{\frac{n-1}{2}} \\
&&\times \left\Vert f^{\lambda _{1}}\right\Vert _{2} \left\Vert g^{\lambda
_{2}}\right\Vert _{2}\,d\lambda _{1}d\lambda _{2} \\
&\leq &C \Bigg( \int_{\mathbb{R}}\int_{\mathbb{R}} \bigg( \sum_{k=0}^{\infty
}\sum_{l=0}^{\infty } \left\vert \varphi _{j}^{\alpha
}\left((2k+n)\left\vert \lambda _{1}\right\vert , (2l+n)\left\vert
\lambda_{2}\right\vert \right) \right\vert \left\vert \lambda _{1}
\right\vert ^{\frac{n+\delta}{2}}k^{\frac{n-1}{2}}\left\vert \lambda
_{2}\right\vert ^{\frac{n+\delta}{2}}l^{\frac{n-1}{2}}\bigg)^2\, d\lambda
_{1}d\lambda _{2} \Bigg)^{\frac{1}{2}} \\
&&\times \Bigg( \int_{\left\vert \lambda _{1}\right\vert \leq 1}
\int_{\left\vert \lambda _{2}\right\vert \leq 1} \left\Vert
f^{\lambda_{1}}\right\Vert ^2_{2} \left\Vert g^{\lambda _{2}}\right\Vert
^2_{2} \left\vert \lambda _{1} \right\vert ^{-\delta} \left\vert \lambda
_{2} \right\vert ^{-\delta}\, d\lambda _{1}d\lambda _{2} \Bigg)^{\frac{1}{2}}
\\
&\leq &C \sum_{k=0}^{\infty }\sum_{l=0}^{\infty } k^{\frac{n-1}{2}}l^{\frac{
n-1}{2}} \bigg( \int_{\mathbb{R}}\int_{\mathbb{R}} \left\vert \varphi
_{j}^{\alpha } \left((2k+n)\left\vert \lambda _{1}\right\vert ,
(2l+n)\left\vert \lambda_{2}\right\vert \right) \right\vert^2 \left\vert
\lambda _{1} \right\vert ^{n+\delta} \left\vert \lambda_{2}\right\vert
^{n+\delta}\, d\lambda _{1}d\lambda _{2} \bigg)^{\frac{1}{2}} \\
&&\times \Bigg( \int_{\left\vert \lambda _{1}\right\vert \leq 1} \left\Vert
f^{\lambda_{1}}\right\Vert ^2_{2} \left\vert \lambda _{1} \right\vert
^{-\delta} \, d\lambda _{1} \Bigg)^{\frac{1}{2}} \Bigg( \int_{\left\vert
\lambda _{2}\right\vert \leq 1} \left\Vert g^{\lambda _{2}}\right\Vert
^2_{2} \left\vert \lambda _{2} \right\vert ^{-\delta}\, d\lambda _{2} \Bigg)
^{\frac{1}{2}} \\
&\leq &C \sum_{k=0}^{\infty }\sum_{l=0}^{\infty } k^{-1-\frac{\delta}{2}
}l^{-1-\frac{\delta}{2}} \bigg( \int_{\mathbb{R}}\int_{\mathbb{R}}
\left\vert \varphi _{j}^{\alpha } \left(\left\vert \lambda _{1}\right\vert ,
\left\vert \lambda_{2}\right\vert \right) \right\vert^2 \left\vert \lambda
_{1} \right\vert ^{n+\delta} \left\vert \lambda_{2}\right\vert ^{n+\delta}\,
d\lambda _{1}d\lambda _{2} \bigg)^{\frac{1}{2}} \\
&&\times \Bigg( \int_{\left\vert \lambda _{1}\right\vert \leq 1} \left\Vert
f^{\lambda_{1}}\right\Vert ^2_{2} \left\vert \lambda _{1} \right\vert
^{-\delta} \, d\lambda _{1} \Bigg)^{\frac{1}{2}} \Bigg( \int_{\left\vert
\lambda _{2}\right\vert \leq 1} \left\Vert g^{\lambda _{2}}\right\Vert
^2_{2} \left\vert \lambda _{2} \right\vert ^{-\delta}\, d\lambda _{2} \Bigg)
^{\frac{1}{2}} \\
&\leq &C2^{-j(\alpha +\frac{1}{2})}\Bigg( \int_{\left\vert \lambda
_{1}\right\vert \leq 1} \left\Vert f^{\lambda_{1}}\right\Vert ^2_{2}
\left\vert \lambda _{1} \right\vert ^{-\delta} \, d\lambda _{1} \Bigg)^{
\frac{1}{2}} \Bigg( \int_{\left\vert \lambda _{2}\right\vert \leq 1}
\left\Vert g^{\lambda _{2}}\right\Vert ^2_{2} \left\vert \lambda _{2}
\right\vert ^{-\delta}\, d\lambda _{2} \Bigg)^{\frac{1}{2}}.
\end{eqnarray*}
Because
\begin{equation*}
T_{j}^{1}(f_{1},g_{1})(\omega )=T_{j}^{\alpha }(f_{1},g_{1})(\omega ),\quad
\omega \in B_{j}(\xi ,\frac{1}{4}),
\end{equation*}
we have
\begin{eqnarray}
&&\left\Vert T_{j}^{1}(f_{1},g_{1})\right\Vert _{L^{\infty }(B_{j}(\xi \,
\frac{1}{4}))}  \label{m0} \\
&=&\left\Vert T_{j}^{\alpha }(f_{1},g_{1})\right\Vert _{L^{\infty
}(B_{j}(\xi \, \frac{1}{4}))}  \notag \\
&\leq &C2^{-j(\alpha +\frac{1}{2})}\Bigg( \int_{\left\vert \lambda
_{1}\right\vert \leq 1} \left\Vert f_1^{\lambda_{1}}\right\Vert ^2_{2}
\left\vert \lambda _{1} \right\vert ^{-\delta} \, d\lambda _{1} \Bigg)^{
\frac{1}{2}} \Bigg( \int_{\left\vert \lambda _{2}\right\vert \leq 1}
\left\Vert g_1^{\lambda _{2}}\right\Vert ^2_{2} \left\vert \lambda _{2}
\right\vert ^{-\delta}\, d\lambda _{2} \Bigg)^{\frac{1}{2}}.  \notag
\end{eqnarray}
Let us consider the integral about $\lambda _{1}$. Note that
\begin{eqnarray*}
\left\Vert f_{1}^{\lambda _{1}}\right\Vert _{2} &\leq &C2^{nj(1+\gamma)}
\left\Vert f_{1}^{\lambda _{1}}\right\Vert _{\infty} \\
&\leq &C2^{j(1+\gamma )\left( \frac{Q}{2}+1\right) } \left\Vert f
\right\Vert_{L^{\infty }(B_{j}(\xi ,\frac{3}{4}))},
\end{eqnarray*}
we obtain
\begin{eqnarray*}
&&\int_{\left\vert \lambda _{1}\right\vert \leq 1} \left\Vert
f_1^{\lambda_{1}} \right\Vert ^2_{2} \left\vert \lambda _{1} \right\vert
^{-\delta}\,d\lambda _{1} \\
&=&\int_{2^{-2j(1+\gamma )}\leq \left\vert \lambda _{1}\right\vert \leq
1}\left\Vert f_1^{\lambda_{1}}\right\Vert ^2_{2} \left\vert \lambda _{1}
\right\vert ^{-\delta}\, d\lambda _{1}+ \int_{\left\vert \lambda
_{1}\right\vert \leq 2^{-2j(1+\gamma )}}\left\Vert
f_1^{\lambda_{1}}\right\Vert ^2_{2} \left\vert \lambda _{1} \right\vert
^{-\delta}\, d\lambda _{1} \\
&\leq &2^{2\delta j(1+\gamma )}\left\Vert f_{1} \right\Vert _{2}^{2} +
C2^{j(1+\gamma )\left(Q+2\right) } \left\Vert f \right\Vert ^2_{L^{\infty
}(B_{j}(\xi ,\frac{3}{4}))} \int_{\left\vert \lambda_{1}\right\vert \leq
2^{-2j(1+\gamma )}} \left\vert \lambda _{1} \right\vert ^{-\delta}\,
d\lambda _{1} \\
&\leq &C2^{j(1+\gamma )\left(Q+2\delta \right)} \left\Vert f \right\Vert
^2_{L^{\infty }(B_{j}(\xi ,\frac{3}{4}))}.
\end{eqnarray*}
In the same way, we have
\begin{equation}
\int_{\left\vert \lambda _{1}\right\vert \leq 1} \left\Vert
g_1^{\lambda_{2}} \right\Vert ^2_{2} \left\vert \lambda _{2} \right\vert
^{-\delta}\,d\lambda _{2} \leq C2^{j(1+\gamma )\left(Q+2\delta \right)}
\left\Vert g \right\Vert ^2_{L^{\infty }(B_{j}(\xi ,\frac{3}{4}))}.
\label{m1}
\end{equation}
From (\ref{m0}) and above estimates, we get
\begin{equation*}
\left\Vert T_{j}^{1}(f_{1},g_{1})\right\Vert _{L^{\infty }(B_{j}(\xi ,\frac{1
}{4} ))}\leq C2^{-j(\alpha +\frac{1}{2})} 2^{j(1+\gamma )\left(Q+2\delta
\right)} \left\Vert f\right\Vert _{L^{\infty}(B_{j}(\xi ,\frac{3}{4}))}
\left\Vert g\right\Vert _{L^{\infty }(B_{j}(\xi ,\frac{3}{4}))}.
\end{equation*}
It follows that
\begin{equation*}
\left\Vert T_{j}^{1}(f_{1},g_{1})\right\Vert _{L^{\infty }} \leq
C2^{-j(\alpha +\frac{1}{2})} 2^{j(1+\gamma )\left(Q+2\delta \right)}
\left\Vert f\right\Vert _{L^{\infty }}\left\Vert g\right\Vert _{L^{\infty }}.
\end{equation*}
Thus, whenever $\alpha >Q- \frac{1}{2}$, we can choose $\gamma, \delta >0$
such that $\alpha >(1+\gamma )(Q+2\delta)- \frac{1}{2}$, which implies there
exists an $\varepsilon >0$ such that
\begin{equation*}
\left\Vert T_{j}^{1}(f_{1},g_{1})\right\Vert _{L^{\infty }}\leq
C2^{-\varepsilon j}\left\Vert f\right\Vert _{L^{\infty }}\left\Vert
g\right\Vert _{L^{\infty }}.
\end{equation*}
The proof of Theorem \ref{infty} is completed.
\end{proof}

\subsection{The point $(2,\infty ,2)$}

\begin{theorem}
\label{Theorem212} If $\alpha >\frac{Q-1}{2}$, then $S^{\alpha }$ is bounded
from $L^{2}(\mathbb{H}^{n})\times L^{\infty }(\mathbb{H}^{n})$ to $L^{2}(
\mathbb{H}^{n})$.
\end{theorem}

\begin{proof}
\quad As above, it suffices to estimate $T_{j}^{1}(f_{1},g_{1})$. We write $
T_{j}^{\alpha }(f,g)$ as
\begin{eqnarray*}
T_{j}^{\alpha }(f,g)(z,t) &=&\int_{0}^{\infty }\int_{0}^{\infty }\varphi
_{j}^{\alpha }(\lambda _{1},\lambda _{2})\sum_{k=0}^{\infty
}(2k+n)^{-n-1}f\ast (\widetilde{e}_{k}^{\lambda _{1}}+\widetilde{e}
_{k}^{-\lambda _{1}})(z,t) \\
&&\times \sum_{l=0}^{\infty }(2l+n)^{-n-1}g\ast (\widetilde{e}_{l}^{\lambda
_{2}}+\widetilde{e}_{l}^{-\lambda _{2}})(z,t)\,d\mu (\lambda _{1})d\mu
(\lambda _{2}) \\
&=&C\int_{-\infty }^{\infty}\int_{-\infty }^{\infty } e^{-i(\lambda
_{1}+\lambda _{2})t}\sum_{k=0}^{\infty }\sum_{l=0}^{\infty }\varphi
_{j}^{\alpha }\left( (2k+n)\left\vert \lambda _{1}\right\vert
,(2l+n)\left\vert \lambda _{2}\right\vert \right) \\
&&\times \left( f^{\lambda _{1}}\ast _{\lambda _{1}}\varphi _{k}^{\lambda
_{1}}\right) (z)\left( g^{\lambda _{2}}\ast _{\lambda _{2}}\varphi
_{l}^{\lambda _{2}}\right) (z)\left\vert \lambda _{1}\right\vert
^{n}\left\vert \lambda _{2}\right\vert ^{n}\,d\lambda _{1}d\lambda _{2} \\
&=&C\int_{-\infty }^{\infty }e^{-i\lambda _{1}t}\int_{-\infty }^{\infty
}\sum_{k=0}^{\infty }\sum_{l=0}^{\infty }\varphi _{j}^{\alpha }\left(
(2k+n)\left\vert \lambda _{1}-\lambda _{2}\right\vert ,(2l+n)\left\vert
\lambda _{2}\right\vert \right) \\
&&\times \left( f^{\lambda _{1}-\lambda _{2}}\ast _{\lambda _{1}-\lambda
_{2}}\varphi _{k}^{\lambda _{1}-\lambda _{2}}\right) (z)\left( g^{\lambda
_{2}}\ast _{\lambda _{2}}\varphi _{l}^{\lambda _{2}}\right) (z)\left\vert
\lambda _{1}-\lambda _{2}\right\vert ^{n}\left\vert \lambda _{2}\right\vert
^{n}\,d\lambda _{1}d\lambda _{2}.
\end{eqnarray*}
Then, applying the Plancherel theorem in variable $t$, Mikowski's
inequality, the orthogonality of the special Hermite projections, (\ref{phi2}) and Plancherel formula (\ref{Plancherel}), we get that
\begin{eqnarray*}
&&\int_{\mathbb{C}^{n}}\int_{\mathbb{R}}\left\vert T_{j}^{\alpha
}(f,g)(z,t)\right\vert ^{2}\,dtdz \\
&=&C\int_{\mathbb{C}^{n}}\int_{\mathbb{R}}\bigg\vert \int_{\mathbb{R}
}e^{-i\lambda _{1}t}\int_{\mathbb{R}}\sum_{l=0}^{\infty }\sum_{k=0}^{\infty
}\varphi _{j}^{\alpha }\big( (2k+n)\left\vert \lambda _{1}-\lambda
_{2}\right\vert ,(2l+n)\left\vert \lambda _{2}\right\vert \big) \\
&& \times \left( f^{\lambda _{1}-\lambda _{2}}\ast _{\lambda _{1}-\lambda
_{2}}\varphi _{k}^{\lambda _{1}-\lambda _{2}}\right) (z)\left( g^{\lambda
_{2}}\ast _{\lambda _{2}}\varphi _{l}^{\lambda _{2}}\right) (z)\left\vert
\lambda _{1}-\lambda _{2}\right\vert ^{n}\left\vert \lambda _{2}\right\vert
^{n}\,d\lambda _{2}d\lambda _{1}\bigg\vert ^{2}\,dtdz \\
&=&C\int_{\mathbb{R}}\int_{\mathbb{C}^{n}}\bigg\vert \int_{\mathbb{R}
}\sum_{l=0}^{\infty }\sum_{k=0}^{\infty } \varphi _{j}^{\alpha }\big(
(2k+n)\left\vert \lambda _{1}-\lambda _{2}\right\vert ,(2l+n)\left\vert
\lambda _{2}\right\vert \big) \\
&& \times \left( f^{\lambda _{1}-\lambda _{2}}\ast _{\lambda _{1}-\lambda
_{2}}\varphi _{k}^{\lambda _{1}-\lambda _{2}}\right) (z)\left( g^{\lambda
_{2}}\ast _{\lambda _{2}}\varphi _{l}^{\lambda _{2}}\right) (z)\left\vert
\lambda _{1}-\lambda _{2}\right\vert ^{n}\left\vert \lambda _{2}\right\vert
^{n}d\lambda _{2}\bigg\vert ^{2}\,dzd\lambda _{1} \\
&\leq &C\int_{\mathbb{R}}\int_{\mathbb{C}^{n}}\Bigg( \int_{\mathbb{R}}
\sum_{l=0}^{\infty } \left\vert \sum_{k=0}^{\infty } \varphi _{j}^{\alpha }
\big( (2k+n)\left\vert \lambda _{1}-\lambda _{2}\right\vert
,(2l+n)\left\vert \lambda _{2}\right\vert \big) \left( f^{\lambda
_{1}-\lambda _{2}}\ast _{\lambda _{1}-\lambda _{2}}\varphi _{k}^{\lambda
_{1}-\lambda _{2}}\right) (z)\right\vert \\
&& \times \left\vert \left( g^{\lambda _{2}}\ast _{\lambda _{2}}\varphi
_{l}^{\lambda _{2}}\right) (z)\right\vert \left\vert \lambda _{1}-\lambda
_{2}\right\vert ^{n}\left\vert \lambda _{2}\right\vert ^{n}\,d\lambda _{2}
\Bigg) ^{2}\,dzd\lambda _{1} \\
&\leq &C\int_{\mathbb{R}}\Bigg( \int_{\mathbb{R}} \sum_{l=0}^{\infty }
\left\Vert \sum_{k=0}^{\infty }\varphi _{j}^{\alpha }\big( (2k+n)\left\vert
\lambda _{1}-\lambda _{2}\right\vert ,(2l+n)\left\vert \lambda
_{2}\right\vert \big) \left( f^{\lambda _{1}-\lambda _{2}}\ast _{\lambda
_{1}-\lambda _{2}}\varphi _{k}^{\lambda _{1}-\lambda _{2}}\right)
\right\Vert _{2} \\
&& \times \left\Vert g^{\lambda _{2}}\ast _{\lambda _{2}}\varphi
_{l}^{\lambda _{2}}\right\Vert _{\infty }\left\vert \lambda _{1}-\lambda
_{2}\right\vert ^{n}\left\vert \lambda _{2}\right\vert ^{n}\,d\lambda _{2}
\Bigg) ^{2}\,d\lambda _{1} \\
&\leq &C\int_{\mathbb{R}}\Bigg( \int_{\mathbb{R}} \sum_{l=0}^{\infty } \left(
\sum_{k=0}^{\infty }\left\vert \varphi _{j}^{\alpha }\big( (2k+n)\left\vert
\lambda _{1}-\lambda _{2}\right\vert ,(2l+n)\left\vert \lambda
_{2}\right\vert \big) \right\vert ^{2}\left\Vert f^{\lambda _{1}-\lambda
_{2}}\ast _{\lambda _{1}-\lambda _{2}}\varphi _{k}^{\lambda _{1}-\lambda
_{2}}\right\Vert _{2}^{2}\right) ^{\frac{1}{2}} \\
&& \times \left\Vert g^{\lambda _{2}}\right\Vert _{2}\left\Vert \varphi
_{l}^{\lambda _{2}}\right\Vert _{2}\left\vert \lambda _{1}-\lambda
_{2}\right\vert ^{n}\left\vert \lambda _{2}\right\vert ^{n}\,d\lambda _{2}
\Bigg) ^{2}\,d\lambda _{1} \\
&\leq &C \Bigg( \int_{\left\vert \lambda _{2}\right\vert \leq 1} \sum_{l\leq
\frac{1}{\left\vert \lambda _{2}\right\vert }} \left\Vert g^{\lambda _{2}}
\right\Vert ^2_{2} \left\Vert \varphi _{l}^{\lambda _{2}} \right\Vert ^2_{2}
\left\vert \lambda _{2}\right\vert ^{2n- \delta}\,d\lambda_{2} \Bigg) \\
&& \times \Bigg( \int_{\mathbb{R}} \int_{\mathbb{R}} \sum_{l=0}^{\infty }
\sum_{k=0}^{\infty }\left\vert \varphi _{j}^{\alpha }\big( (2k+n)\left\vert
\lambda _{1}-\lambda _{2}\right\vert ,(2l+n)\left\vert \lambda
_{2}\right\vert \big) \right\vert ^{2} \\
&& \times \left\Vert f^{\lambda _{1}-\lambda _{2}}\ast _{\lambda
_{1}-\lambda _{2}}\varphi _{k}^{\lambda _{1}-\lambda _{2}}\right\Vert
_{2}^{2} \left\vert \lambda _{1}-\lambda _{2}\right\vert ^{2n} \left\vert
\lambda _{2}\right\vert ^{\delta}\,d\lambda _{2} d\lambda _{1} \Bigg) \\
&= &C \Bigg( \int_{\left\vert \lambda _{2}\right\vert \leq 1} \sum_{l\leq
\frac{1}{\left\vert \lambda _{2}\right\vert }} \left\Vert g^{\lambda _{2}}
\right\Vert ^2_{2} \left\Vert \varphi _{l}^{\lambda _{2}} \right\Vert ^2_{2}
\left\vert \lambda _{2}\right\vert ^{2n- \delta}\,d\lambda_{2} \Bigg) \\
&& \times \Bigg( \sum_{k=0}^{\infty } \int_{\mathbb{R}} \sum_{l=0}^{\infty }
\int_{\mathbb{R}} \left\vert \varphi _{j}^{\alpha }\big( (2k+n)\left\vert
\lambda _{1}\right\vert ,(2l+n)\left\vert \lambda _{2}\right\vert \big)
\right\vert ^{2} \left\Vert f^{\lambda _{1}}\ast _{\lambda _{1}}\varphi
_{k}^{\lambda _{1}}\right\Vert _{2}^{2} \left\vert \lambda _{1}\right\vert
^{2n} \left\vert \lambda _{2}\right\vert ^{\delta}\,d\lambda _{2} d\lambda
_{1} \Bigg) \\
&\leq &C \Bigg( \int_{\left\vert \lambda _{2}\right\vert \leq 1} \left\Vert
g^{\lambda _{2}} \right\Vert ^2_{2} \left\vert \lambda _{2}\right\vert ^{n-
\delta} \bigg( \sum_{l\leq \frac{1}{\left\vert \lambda _{2}\right\vert }}
l^{n-1} \bigg)\,d\lambda_{2} \Bigg) \\
&& \times \Bigg( \sum_{k=0}^{\infty } \int_{\mathbb{R}} \left\Vert
f^{\lambda _{1}}\ast _{\lambda _{1}}\varphi _{k}^{\lambda _{1}}\right\Vert
_{2}^{2} \left\vert \lambda _{1}\right\vert ^{2n}\bigg( \sum_{l=0}^{\infty }
(2l+n)^{-1- \delta} \bigg) \int_{\mathbb{R}} \left\vert \varphi _{j}^{\alpha
}\big( (2k+n)\left\vert \lambda _{1}\right\vert , \left\vert \lambda
_{2}\right\vert \big) \right\vert ^{2} \left\vert \lambda _{2}\right\vert
^{\delta}\,d\lambda _{2} d\lambda _{1} \Bigg) \\
&\leq &C2^{-j(2\alpha +1)} \Bigg( \int_{\left\vert \lambda _{2}\right\vert
\leq 1} \left\Vert g^{\lambda _{2}} \right\Vert ^2_{2} \left\vert \lambda
_{2}\right\vert ^{- \delta}\,d\lambda_{2} \Bigg) \Bigg( \sum_{k=0}^{\infty }
\int_{\mathbb{R}} \left\Vert f^{\lambda _{1}}\ast _{\lambda _{1}}\varphi
_{k}^{\lambda _{1}}\right\Vert _{2}^{2} \left\vert \lambda _{1}\right\vert
^{2n}\, d\lambda _{1} \Bigg) \\
&\leq &C2^{-j(2\alpha +1)} \left\Vert f\right\Vert _{2}^{2} \int_{\left\vert
\lambda _{2}\right\vert \leq 1} \left\Vert g^{\lambda _{2}} \right\Vert
^2_{2} \left\vert \lambda _{2}\right\vert ^{- \delta}\,d\lambda_{2} .
\end{eqnarray*}
It follows from (\ref{m1}) that
\begin{eqnarray*}
&&\left\Vert T_{j}^{1}(f_{1},g_{1})\right\Vert ^2_{L^{2}(B_{j}(\xi ,\frac{1}{
4}))} \\
&=&\left\Vert T_{j}^{\alpha }(f_{1},g_{1})\right\Vert ^2_{L^{2} (B_{j}(\xi ,
\frac{1}{4}))} \\
&\leq &C2^{-j(2\alpha +1)} \left\Vert f_{1}\right\Vert _{2}^{2}
\int_{\left\vert \lambda _{2}\right\vert \leq 1} \left\Vert g_{1}^{\lambda
_{2}} \right\Vert ^2_{2} \left\vert \lambda _{2}\right\vert ^{-
\delta}\,d\lambda_{2} \\
&\leq &C2^{-j(2\alpha +1)} 2^{j(1+\gamma )\left(Q+2 \delta \right)}
\left\Vert f\right\Vert ^2_{L^{2}(B_{j}(\xi ,\frac{3}{4}))} \left\Vert g
\right\Vert ^2_{L^{\infty }(B_{j}(\xi ,\frac{3}{4}))}.
\end{eqnarray*}
Taking the $L^{2}$ norm with respect to $\xi $ yields that
\begin{equation*}
\left\Vert T_{j}^{1}(f_{1},g_{1})\right\Vert _{2}\leq C2^{-j(\alpha + \frac{1
}{2})} 2^{j(1+\gamma )(\frac{Q}{2}+ \delta)}\left\Vert f \right\Vert _{2}
\left\Vert g \right\Vert _{\infty }.
\end{equation*}
Thus, whenever $\alpha > \frac{Q-1}{2}$, we can choose $\gamma, \delta >0$
such that $\alpha > (1+\gamma )(\frac{Q}{2}+ \delta)- \frac{1}{2}$, which
yields that there exists $\varepsilon >0$ such that
\begin{equation*}
\left\Vert T_{j}^{1}(f_{1},g_{1})\right\Vert _{2}\leq C2^{-\varepsilon
j}\left\Vert f\right\Vert _{2}\left\Vert g\right\Vert _{\infty }.
\end{equation*}
The proof of Theorem \ref{Theorem212} is completed.
\end{proof}

\section*{Appendix: bilinear interpolation}

Because $p_{1}$ and $p_{2}$ are symmetric, we have obtained, in two sections
above, the boundedness of the bilinear Riesz means $S^{\alpha }$ at some
specific triples of points $(p_{1},p_{2},p)$ like
\begin{equation*}
(1,1,\frac{1}{2}), (2,2,1), (\infty,\infty,\infty), (1,2,\frac{2}{3}), (2,1,
\frac{2}{3}),(1,\infty ,1),(\infty,1,1),(2,\infty,2),(\infty,2,2)\text{.}
\end{equation*}
We can obtain the intermediate boundedness of $S^{\alpha}$ by using of the
bilinear interpolation via complex method adapted to the setting of analytic
families or real method in \cite{Graf}. Bernicot et al. \cite{Bern}
described how to make use of real method. The full results in Main Theorem
are proved in this way. We outline this argument for reader's convenience.

Consider a spherical decomposition of $S^{\alpha }$ as
\begin{equation*}
S^{\alpha }=\sum_{j=0}^{\infty }2^{-j\alpha }T_{j,\alpha },
\end{equation*}
where
\begin{equation*}
T_{j,\alpha }(f,g)=\int_{0}^{\infty }\int_{0}^{\infty }\varphi _{j,\alpha
}\left( \lambda _{1},\lambda _{2}\right) P_{\lambda _{1}}fP_{\lambda
_{2}}gd\mu (\lambda _{1})d\mu (\lambda _{2})
\end{equation*}
and
\begin{equation*}
\varphi _{j,\alpha }\left( s,t\right) =2^{j\alpha }(1-s-t)_{+}^{\alpha
}\varphi (2^{j}\left( 1-s-t\right) ).
\end{equation*}
In the preceding sections, we have actually obtained the estimates of the
form
\begin{equation}
\left\Vert T_{j,\alpha }\right\Vert _{L^{p_{1}}\times L^{p_{2}}\rightarrow
L^{p}}\leq C2^{j\alpha (p_{1},p_{2})}  \label{Tj}
\end{equation}
at some triples of points $(p_{1},p_{2},p)$. Since $\alpha (p_{1,}p_{2})$
only depends on the point $(p_{1,}p_{2},p)$, (\ref{Tj}) also holds for any
other $T_{j,\alpha ^{\prime }}$, i.e.,
\begin{equation*}
\left\Vert T_{j,\alpha ^{\prime }}\right\Vert _{L^{p_{1}}\times
L^{p_{2}}\rightarrow L^{p}}\leq C2^{j\alpha (p_{1},p_{2})}.
\end{equation*}
So, fixing $j$ and $\alpha ^{\prime }$ and applying bilinear real
interpolation theorem on $T_{j,\alpha ^{\prime }}$, we can conclude that if
the point $(p_{1},p_{2},p)$ satisfies
\begin{equation*}
\left( \frac{1}{p_{1}},\frac{1}{p_{2}},\frac{1}{p}\right) =(1-\theta )\left(
\frac{1}{p_{1}^{0}},\frac{1}{p_{2}^{0}},\frac{1}{p^{0}}\right) +\theta
\left( \frac{1}{p_{1}^{1}},\frac{1}{p_{2}^{1}},\frac{1}{p^{1}}\right)
\end{equation*}
for some $\theta \in (0,1)$ and $%
(p_{1}^{0},p_{2}^{0},p^{0}),(p_{1}^{1},p_{2}^{1},p^{1})$, \thinspace we have
that
\begin{equation*}
\left\Vert T_{j,\alpha ^{\prime }}\right\Vert _{L^{p_{1}}\times
L^{p_{2}}\rightarrow L^{p}}\leq C2^{j((1-\theta )\alpha
(p_{1}^{0},p_{2}^{0})+\theta \alpha (p_{1}^{1},p_{2}^{1}))}.
\end{equation*}
Define $\alpha (p_{1,}p_{2})=(1-\theta )\alpha (p_{1}^{0},p_{2}^{0})+\theta
\alpha (p_{1}^{1},p_{2}^{1})$ and let $\alpha ^{\prime }=\alpha $. It
follows that
\begin{equation*}
\left\Vert S^{\alpha }\right\Vert _{L^{p_{1}}\times L^{p_{2}}\rightarrow
L^{p}}\leq \sum_{j=0}^{\infty }2^{-j\alpha }\left\Vert T_{j,\alpha
}\right\Vert _{L^{p_{1}}\times L^{p_{2}}\rightarrow L^{p}}\leq
C\sum_{j=0}^{\infty }2^{-j\alpha }2^{j\alpha (p_{1,}p_{2})}.
\end{equation*}
Thus, when $\alpha >\alpha (p_{1,}p_{2})$, we have $\left\Vert S^{\alpha
}\right\Vert _{L^{p_{1}}\times L^{p_{2}}\rightarrow L^{p}}\leq C$, i.e., $
S^{\alpha }$ is bounded from $L^{p_{1}}(\mathbb{H}^{n})\times L^{p_{1}}(
\mathbb{H}^{n})$ to $L^{p}(\mathbb{H}^{n})$. \vskip 0.5 cm

\textbf{Acknowledgements} {\quad The first author is supported by National
Natural Science Foundation of China (Grant No. 11371036). The second author
is supported by China Scholarship Council (Grant No. 201606010026). }


\begin{thebibliography}{99}
\bibitem{Bern} F. Bernicot, L. Grafakos, L. Song, L. Yan, \textit{The
bilinear Bochner-Riesz problem}, J. Anal. Math. 127 (2015), 179--217.

\bibitem{Feff} C. Feffermann, \textit{Inequalities for strongly singular
convolution operators}, Acta Math. 124 (1970), 9--36.

\bibitem{Graf} L. Grafakos, L. Liu, S. Lu and F. Zhao, \textit{The
multilinear Marcinkiewicz interpolation theorem revisited: the behavior of
the constant}, J. Funct. Anal. 262 (2012), no. 5, 2289--2313.

\bibitem{Mauc} G. Mauceri, \textit{Riesz means for the eigenfunction
expansions for a class of hypoelliptic differential operators}, Ann. Inst.
Fourier (Grenoble). 31 (1981), no. 4, v, 115--140.

\bibitem{Mull} D. M\"{u}ller, \textit{A restriction theorem for the
Heisenberg group}, Ann. Math(2). 131 (1990), no. 3, 567--587.

\bibitem{Mull2} D. M\"{u}ller, \textit{On Riesz means of
eigenfunction expansions for the Kohn-Laplacian}, J. Reine Angew. Math. 401 (1989), 113--121.

\bibitem{Stein} E. M. Stein, \textit{Interpolation of linear operators}, Trans. Amer. Math. Soc. 83 (1956), 482--492.

\bibitem{Strich} R. S. Strichartz, \textit{Harmonic analysis as spectral
theory of Laplacians}, J. Funct. Anal. 87 (1989), no. 1, 51--148.

\bibitem{Strich2} R. S. Strichartz, $L^{p}$\textit{\ harmonic analysis and
Radon transform on the Heisenberg group}, J. Funct. Anal. (96) 1991, no.2,
350--406.

\bibitem{Thang} S. Thangavelu, Harmonic analysis on the Heisenberg group.
Progress in Mathematics, 159.
\end{thebibliography}
\end{document}